\documentclass[11pt]{amsart}

\usepackage{amsmath,amssymb,latexsym,soul,cite,mathrsfs}

\usepackage{color,enumitem,graphicx}
\usepackage[colorlinks=true,urlcolor=blue,
citecolor=red,linkcolor=red,linktocpage,pdfpagelabels,
bookmarksnumbered,bookmarksopen]{hyperref}
\usepackage[english]{babel}

\usepackage[left=2.9cm,right=2.9cm,top=2.8cm,bottom=2.8cm]{geometry}
\usepackage[hyperpageref]{backref}

\usepackage[colorinlistoftodos]{todonotes}
\makeatletter
\providecommand\@dotsep{5}
\def\listtodoname{List of Todos}
\def\listoftodos{\@starttoc{tdo}\listtodoname}
\makeatother

\numberwithin{equation}{section}

\newcommand{\Om} {\Omega}

\newtheorem{Theorem}{Theorem}[section]
\newtheorem{Lemma}[Theorem]{Lemma}

\newtheorem{Corollary}[Theorem]{Corollary}
\newtheorem{Remark}[Theorem]{Remark}
\newtheorem{Definition}[Theorem]{Definition}

\begin{document}

\title[Mixed local and nonlocal quasilinear problems]
{Mixed local and nonlocal Sobolev inequalities with extremal and associated quasilinear singular elliptic problems}
\author{Prashanta Garain and Alexander Ukhlov}

\address[Prashanta Garain ]
{\newline\indent Department of Mathematics
	\newline\indent
Ben-Gurion University of the Negev
	\newline\indent
P.O.B. 653
\newline\indent
Beer Sheva 8410501, Israel
\newline\indent
Email: {\tt pgarain92@gmail.com} }

\address[Alexander Ukhlov]
{\newline\indent Department of Mathematics
\newline\indent
Ben-Gurion University of the Negev
\newline\indent
P.O.B.-653
\newline\indent
Beer Sheva 8410501, Israel
\newline\indent
Email: {\tt ukhlov@math.bgu.ac.il} }

\pretolerance10000

\begin{abstract}
In this article, we consider mixed local and nonlocal Sobolev $(q,p)$-inequalities with extremal in the case $0<q<1<p<\infty$. We prove  that the extremal  of such inequalities is unique up to a multiplicative constant that is associated with a singular elliptic problem involving the mixed local and nonlocal $p$-Laplace operator. Moreover, it is proved that the mixed Sobolev inequalities are  necessary and sufficient condition for the existence of weak solutions of such singular problems. As a consequence, a relation between the singular $p$-Laplace and mixed local and nonlocal $p$-Laplace equation is established. Finally, we investigate the existence, uniqueness, regularity and symmetry properties of weak solutions for such problems. 
\end{abstract}


\maketitle

\section{Introduction}
Sharp Sobolev inequalities in both the local and nonlocal cases have been extensively studied in the literature to date and have a wide range of applications in the theory of partial differential equations. Sharp constants and the extremal functions of such inequalities  plays a crucial role (see for instance, Aubin \cite{Aubin}, Talenti \cite{Talenti}, Lieb \cite{Lieb}, Saloff-Coste \cite{Scoste}, Swanson \cite{Swanson}, Di Nezza-Palatucci-Valdinoci \cite{Hitchhiker'sguide} and the references therein) in geometric analysis and continuum mechanics \cite{PolyaSzego}. Recall the classical Sobolev $(q,p)$-inequality: Let $\Omega$ be a bounded domain in $\mathbb R^N$, then  
\begin{equation}\label{clssb}
S\left(\int_{\Om}|u|^q\,dx\right)^\frac{1}{q}\leq\left(\int_{\Om}|\nabla u|^p\,dx\right)^\frac{1}{p},\,\,1<q<p^{*}:={Np}/{(N-p)},
\end{equation}
for any function $u\in W_0^{1,p}(\Om)$, $1\leq p<N$, where $S$ is the Sobolev constant. The best constant 
\begin{equation}\label{clssbbc}
\lambda_q(\Om):=\inf_{0\neq u\in W_0^{1,p}(\Om)\setminus\{0\}}\frac{\int_{\Om}|\nabla u|^p\,dx}{\left(\int_{\Om}|u|^q\,dx\right)^\frac{p}{q}}
\end{equation}
of the inequality \eqref{clssb} is achieved at a solution $u_q\in W_0^{1,p}(\Om)$ of the following Dirichlet boundary value problem (see \^{O}tani \cite{Otani}, Franzina-Lamberti \cite{FL})
\begin{equation}\label{clssbeqn}
-\Delta_p u:=-\text{div}(|\nabla u|^{p-2}\nabla u)=\lambda_q(\Om)\|u\|_{L^q(\Om)}^{p-q}|u|^{q-2}u\text{ in }\Om,\,u=0\text{ on }\partial\Om,
\end{equation}

In the case $q=p$, by the Min-Max Principle, the quantity $\lambda_p(\Om)$ is equal to the first eigenvalue of the $p$-Laplace operator with the Dirichlet boundary condition, which is widely studied in the literature, see Lindqvist \cite{PLin2}, Belloni-Kawohl \cite{BK} and the references therein. 

In the case $q=p^*$ and $\Om=\mathbb{R}^N$ the minimization problem~(\ref{clssbbc}) was solved using symmetrization, we refer to Aubin \cite{Aubin}, Talenti \cite{Talenti} and the references therein. In the nonlocal case, see Di Nezza-Palatucci-Valdinoci \cite{Hitchhiker'sguide}, Lindgren-Lindqvist \cite{Ling}, Brasco-Lindgren-Parini \cite{BLP}, Brasco-Parini \cite{BrPr}, Franzina-Palatucci \cite{FP}, Nguyen-D\'{\i}az-Quoc-Hung \cite{Diaz} and the references therein.

For $0<q<1$, in the local case, the inequality \eqref{clssb} has been studied by Anello-Faraci-Iannizzotto \cite{GFA} for any $1<p<\infty$, where they proved that
\begin{equation}\label{lclssbbc}
\nu(\Om):=\inf_{u\in W_0^{1,p}(\Om)\setminus\{0\}}\left\{\int_{\Om}|\nabla u|^p\,dx:\int_{\Om}|u|^q\,dx=1\right\}
\end{equation}
is achieved at a solution $u_q\in W_0^{1,p}(\Om)$ of the singular $p$-Laplace equation,
\begin{equation}\label{lclssbeqn}
-\Delta_p u=\nu(\Om)u^{q-1}\text{ in }\Om,\,u>0\text{ in }\Om,\,u=0\text{ on }\partial\Om.
\end{equation}

In the nonlocal context, for any $0<q<1$, Sobolev type inequality in fractional Sobolev spaces is investigated by Ercole-Pereira \cite{EP1} for $u\in W_0^{s,p}(\Om)$ with $0<s<1<p<\infty$. Indeed, for a given nonnegative function $f\in L^m(\Om)\setminus\{0\}$ with some $m\geq 1$, they obtained that
\begin{equation}\label{lclssbbc1}
\zeta(\Om):=\inf_{u\in W_0^{s,p}(\Om)\setminus\{0\}}\left\{\int_{\Om}\frac{|u(x)-u(y)|^p}{|x-y|^{N+ps}}\,dx dy:\int_{\Om}|u|^q f\,dx=1\right\}
\end{equation}
is achieved at a solution $u_q\in W_0^{s,p}(\Om)$ of the singular fractional $p$-Laplace equation,
\begin{equation}\label{lclssbeqn1}
(-\Delta_p)^s u=\zeta(\Om)f(x)u^{q-1}\text{ in }\Om,\,u>0\text{ in }\Om,\,u=0\text{ in }\mathbb{R}^N\setminus\Om,
\end{equation}
where $(-\Delta_p)^s$ is the fractional $p$-Laplace operator given by 
\begin{equation}\label{fplap}
(-\Delta_p)^s u:=\text{P.V.}\int\limits_{\mathbb{R}^N}\frac{|u(x)-u(y)|^{p-2}\big(u(x)-u(y)\big)}{|x-y|^{N+ps}}\,dy,
\end{equation}
where P.V. denotes the principal value. See also Ercole-Pereira-Sanchis \cite{EPS,EP} and Bal-Garain \cite{BGmm} for related works.

In this article, one of our main purpose is to study the following mixed minimizing problem in a bounded Lipschitz domain $
\Omega\subset\mathbb{R}^N$, $N\geq 2$,
\begin{equation}\label{lclssbbc2}
\mu(\Om):=\inf_{u\in W_0^{1,p}(\Om)\setminus\{0\}}\left\{\int_{\Om}|\nabla u|^p\,dx+\int_{\Om}\frac{|u(x)-u(y)|^p}{|x-y|^{N+ps}}\,dx dy:\int_{\Om}|u|^{1-\delta} f\,dx=1\right\},
\end{equation}
where $0<\delta<1<p<\infty$, $0<s<1$ and $f$ is some nonnegative function belongs to $L^m(\Om)\setminus\{0\}$, $m\geq 1$,  (See Theorem \ref{thm5}). We prove that $\mu(\Om)$ is achieved at a solution of the following mixed local and nonlocal $p$-Laplace equation,
$$
-\Delta_p u+(-\Delta_p)^s u=\frac{f(x)}{u(x)^\delta}\text{ in }\Om,\,u>0\text{ in }\Om,\,u=0\text{ in }\mathbb{R}^N\setminus\Om.\leqno{\mathcal{(S)}}
$$

Further, we investigate the problem $\mathcal{(S})$ for any $\delta>0$. Let us discuss some related results. Due to fact that $\delta>0$, the nonlinearity in the problem $\mathcal{(S)}$ blows up near the origin. Such property is referred to as singularity and studied widely over the last three decades in the context of both local and nonlocal cases separately. We cite here some related works with no intent to furnish an exhaustive list.

In particular, for the local case, with a given nonnegative integrable function $f$, the singular $p$-Laplace equation
$$
-\Delta_p u=\frac{f(x)}{u(x)^\delta}\text{ in }\Omega,\,u>0\text{ in }\Om,\,u=0\text{ on }\partial\Om,\leqno{(\mathcal{P})}
$$
has been studied in the semilinear case $p=2$ by Crandall-Rabinowitz-Tartar \cite{CRT} followed by Boccardo-Orsina \cite{BocOr}, Alves-Goncalves-Maia \cite{Alves}, Canino-Degiovanni-Grandinetti-Sciunzi \cite{CaninoSci, Caninovar, Caninouni}. De Cave \cite{DeCave}, Canino-Sciunzi-Trombetta \cite{Canino} and Garain \cite{G} considered the quasilinear case to obtain existence, uniqueness and regularity results. In the perturbed  singular case, multiplicity results for the problem $\mathcal{(P)}$ have been obtained by Arcoya-M\'erida \cite{AM}, Arcoya-Boccardo \cite{arcoya} in the semilinear case; Giacomoni-Schindler-Tak\'{a}\v{c} \cite{GST}, Bal-Garain-Mukherjee \cite{BG, GMwgt} in the quasilinear case and the references therein.

In the nonlocal case, for $0<s<1$, the problem
$$
(-\Delta_p)^s u=\frac{f(x)}{u(x)^\delta}\text{ in }\Omega,\,u>0\text{ in }\Om,\,u=0\text{ in }\mathbb{R}^N\setminus\Om,\leqno{(\mathcal{F})}
$$
has been studied by Fang \cite{Fang} in the semilinear case $p=2$; Canino-Montoro-Sciunzi-Squassina \cite{Caninononloc}, Garain-Mukherjee \cite{GMnonloc} in the quasilinear case. The perturbed  singular case is investigated by Barrios-De Bonis-Medina-Peral \cite{BDMP}, Adimurthi-Giacomoni-Santra \cite{Adi} for $p=2$; Mukherjee-Sreenadh \cite{MS} in the quasilinear case and the references therein. For measure data problems, see De Cave-Durastanti-Oliva \cite{DeCave2}, Oliva-Orsina-Petitta \cite{OP2, OP3, OP1} and the references therein. We also refer to the nice survey by Ghergu-R\u{a}dulescu \cite{GRbook} on singular elliptic problems.

We emphasis that problems with mixed local and nonlocal characters have been far less studied in the literature to date even for the nonsingular case. The following type of problem,
\begin{equation}\label{mnsin}
-\Delta_p u+(-\Delta_p)^s u=g\text{ in }\Om,
\end{equation}
where $g$ is nonsingular have been studied recently. In this regard, Del Pezzo-Ferreira-Rossi \cite{Rossi1} studied the following eigenvalue problem
\begin{equation}\label{mevp}
-\Delta_p u-\int\limits_{\mathbb{R}^N}\mathcal{J}(x,y)|u(x)-u(y)|^{p-2}(u(x)-u(y))\,dy=\lambda|u|^{p-2}u\text{ in }\Om,\,u=0\text{ in }\mathbb{R}^N\setminus\Om,
\end{equation}
where $\mathcal{J}:\mathbb{R}^N\to\mathbb{R}^+$ is a nonnegative, nonsingular, radially symmetric and compactly supported kernel. For related results, see also da Silva-Salort \cite{Silvas}, Biagi-Dipierro-Valdinoci-Vecchi \cite{Biagi2} and the references therein. 

The question of regularity related to the problem \eqref{mnsin} has been settled by Barlow-Bass-Chen-Kassmann \cite{BBCK}, Chen-Kim-Song-Vondra\v{c}ek \cite{Song2, Song1}, Garain-Kinnunen \cite{GK, GK1} and the references therein. Furthermore, existence, regularity, symmetry and maximum principles along various qualitative properties of solutions have been recently investigated by Biagi-Dipierro-Valdinoci-Vecchi \cite{BSVV2, BSVV1}, Dipierro-Proietti Lippi-Valdinoci \cite{DPV20, DPV21}, Biagi-Mugnai-Vecchi \cite{Biagi1}, Dipierro-Ros-Oton-Serra-Valdinoci \cite{DRXJV20}. 

\subsection*{Our contribution:} We start with the existence results (Theorem \ref{thm1}-\ref{thm3}) for the problem $\mathcal{(S)}$, where the idea stem from the approximation approach in Boccardo-Orsina \cite{BocOr}. To this end, employing the idea from Boccardo-Murat \cite{BocMu}, a gradient convergence theorem (Theorem \ref{grad}) is established in the mixed case that is very crucial to pass the limit in obtaining our existence results. It turns out that for $0<\delta\leq 1$, a solution $u$ lies in $W_0^{1,p}(\Om)$. Although it is worth mentioning that the situation is different for $\delta>1$, since here we only find solutions $u$ in $W^{1,p}_{\mathrm{loc}}(\Om)$ such that $u^\frac{\delta+p-1}{p}\in W_0^{1,p}(\Om)$. Such phenomena is referred to as the boundary condition of our solutions. Indeed, a more general boundary condition is considered in Definition \ref{bc}. Under this definition, a uniqueness result (Theorem \ref{thm4}) is proved. We emphasis that to deal with our general boundary condition, the solutions may not belong to $W_0^{1,p}(\Om)$ and therefore, we cannot apply the density arguments as in Lemma \ref{testfn} that is an important ingredient to prove the uniqueness result in Corollary \ref{tstcor}. We overcome this difficulty by establishing a comparison principle (see Lemma \ref{cp}) for subsolutions and supersolutions of the problem $\mathcal{(S)}$ following the idea from Canino-Sciunzi \cite{Caninouni}. As a consequence, we deduce a symmetry result, that in particular gives radial solution, provided the given data is radial (Theorem \ref{sym}). Further, we obtain a regularity result in Theorem \ref{regthm}.

Next, we establish a mixed Sobolev inequality with an extremal (Theorem \ref{thm5}). It is shown that such extremals are simple in the sense that they are equal upto a multiplicative constant. Moreover, extremals are associated with the obtained solutions of $\mathcal{(S)}$. We also observe that the mixed Sobolev inequality is a necessary and sufficient condition (Theorem \ref{nsthm}) for the existence of solutions for the problem $\mathcal{(S)}$ and an analogous result holds in the separate local and nonlocal cases (Theorem \ref{nsthm1}-\ref{nsthm2}). As a byproduct, we prove that the singular mixed local and nonlocal $p$-Laplace equation $\mathcal{(S)}$ admits a weak solution if and only if the singular $p$-Laplace equation $\mathcal{(P)}$ does so (Corollary \ref{nsthmcor}). As far as we are aware, our main results are new even in the semilinear case $p=2$.

\subsection*{Organization:} The rest of the article is organized as follows: In Section $2$, we discuss some preliminaries in our setting and state our main results. In Sections $3$, $4$ and $5$, we establish several auxiliary lemmas that are crucial to prove the existence along with regularity, uniqueness and mixed Sobolev inequality respectively. Finally, in Section $6$, we prove our main results and in Section $7$, a gradient convergence theorem is proved for the mixed case respectively.

\subsection*{Notations:} Throughout the paper, we shall use the following notations:
\begin{itemize}
\item For $k\in\mathbb{R}$, we denote by $k^+=\max\{k,0\}$, $k^-=\max\{-k,0\}$ and $k_-=\min\{k,0\}$.
\item $\langle,\rangle$ denotes the standard inner product in $\mathbb{R}^N$.
\item For $r>1$, we denote by $r'=\frac{r}{r-1}$ to mean the conjugate exponent of $r$.
\item For a function $F$ defined over a set $S$ and some constants $c$ and $d$, by $c\leq F\leq d$ in $S$, we mean $c\leq F\leq d$ almost everywhere in $S$. 
\item $C$ will denote a constant which may vary from line to line or even in the same line. If $C$ depends on the parameters $r_1,r_2,\ldots,$ then we write $C=C(r_1,r_2,\ldots,)$.
\end{itemize}

\section{Preliminaries and main results}
In this section, we present some known results for the fractional Sobolev spaces, see \cite{Hitchhiker'sguide} for more details.
Let 
$E\subset \mathbb{R}^N$
be a measurable set. Recall that the Lebesgue space 
$L^{p}(E),1\leq p<\infty$, 
is defined as the space of $p$-integrable functions $u:E\to\mathbb{R}$ with the finite norm
$$ \|u\|_{L^p(E)}=
\left(\int\limits_{E}|u(x)|^p~dx\right)^{1/p}.
$$
By 
$L^p_{\mathrm{loc}}(E)$ 
we denote the space of locally $p$-integrable functions, means that,
$u\in L^p_{\mathrm{loc}}(E)$ 
if and only if  
$u\in L^p(F)$ 
for every compact subset 
$F\subset E$. In the case $0<p<1$, we denote by $L^{p}(E)$ a set of measurable functions such that 
$\int_{E}|u(x)|^p~dx<\infty$.

If $\Omega$ is an open subset of $\mathbb R^N$, the Sobolev space $W^{1,p}(\Omega)$, $1\leq p<\infty$, is defined 
as a Banach space of locally integrable weakly differentiable functions
$u:\Omega\to\mathbb{R}$ equipped with the following norm: 
\[
\|u\|_{W^{1,p}(\Omega)}=\| u\|_{L^p(\Omega)}+\|\nabla u\|_{L^p(\Omega)}.
\]
The space $W^{1,p}_0(\Omega)$ is defined as a closure of a space $C^{\infty}_c(\Omega)$ of smooth functions with compact support in the norm of the Sobolev space $W^{1,p}(\Omega)$. 

\begin{Definition}
Let $\Om\subset\mathbb{R}^N$ with $N\geq 2$ be any open set. The fractional Sobolev space $W^{s,p}(\Omega)$, $0<s<1<p<\infty$, is defined by
$$
W^{s,p}(\Omega)=\Big\{u\in L^p(\Omega):\frac{|u(x)-u(y)|}{|x-y|^{\frac{N}{p}+s}}\in L^p(\Omega\times \Omega)\Big\}
$$
and endowed with the norm
$$
\|u\|_{W^{s,p}(\Omega)}=\left(\int_{\Omega}|u(x)|^p\,dx+\int_{\Omega}\int_{\Omega}\frac{|u(x)-u(y)|^p}{|x-y|^{N+ps}}\,dx\,dy\right)^\frac{1}{p}.
$$
{The fractional Sobolev space with zero boundary values is defined by}
$$
W_{0}^{s,p}(\Omega)={\big\{u\in W^{s,p}(\mathbb{R}^N):u=0\text{ on }\mathbb{R}^N\setminus\Omega\big\}}.
$$
\end{Definition}
Both $W^{s,p}(\Omega)$ and $W_{0}^{s,p}(\Omega)$ are reflexive Banach spaces, see \cite{Hitchhiker'sguide}. The space $W^{s,p}_{\mathrm{loc}}(\Omega)$ is defined analogously.

For the rest of the article, we shall assume that $\Om\subset\mathbb{R}^N$ with $N\geq 2$ is a Lipschitz bounded domain, $0<s<1<p<\infty$, $\delta>0$ and denote by $W^{1,p}(\Om)$ to mean the classical Sobolev space with first weak derivatives, unless otherwise mentioned. The next result asserts that the Sobolev space $W^{1,p}(\Om)$ is continuously embedded in the fractional Sobolev space, see \cite[Proposition $2.2$]{Hitchhiker'sguide}
\begin{Lemma}\label{defineq}
There exists a constant $C=C(N,p,s)>0$ such that
$$
\|u\|_{W^{s,p}(\Om)}\leq C\|u\|_{W^{1,p}(\Om)},\quad\forall\,u\in W^{1,p}(\Omega).
$$
\end{Lemma}

Next, we have the following result from \cite[Lemma $2.1$]{Silva}.

\begin{Lemma}\label{locnon1}
There exists a constant $C=C(N,p,s,\Omega)$ such that
\begin{equation}\label{locnonsem}
\int\limits_{\mathbb{R}^N}\int\limits_{\mathbb{R}^N}\frac{|u(x)-u(y)|^p}{|x-y|^{N+ps}}\,dx\,dy\leq C\int_{\Omega}|\nabla u|^p\,dx,\quad\forall\,u\in W_0^{1,p}(\Omega).
\end{equation}
\end{Lemma}

\begin{Remark}\label{norm12}
Using Lemma \ref{embedding} and (\ref{locnonsem}) the following norm on the space $W_0^{1,p}(\Om)$ defined by
\begin{equation}\label{norm}
\|u\|:=\|u\|_{W_0^{1,p}(\Om)}=\left(\int_{\Omega}|\nabla u|^p\,dx+\int\limits_{\mathbb{R}^N}\int\limits_{\mathbb{R}^N}\frac{|u(x)-u(y)|^p}{|x-y|^{N+ps}}\,dx\,dy\right)^\frac{1}{p}
\end{equation}
is equivalent to the norm 
\begin{equation}\label{equinorm}
\|u\|:=\|u\|_{W_0^{1,p}(\Om)}=\left(\int_{\Omega}|\nabla u|^p\,dx\right)^\frac{1}{p}.
\end{equation}
\end{Remark}

In section $3$, we will use the gradient norm \eqref{equinorm} and in section $5$ we will use the mixed norm \eqref{norm}. In other sections one can consider any one of them unless otherwise mentioned.

For the following Sobolev embedding, see, for example, \cite{Evans, Mazya}.

\begin{Lemma}\label{embedding}
The embedding operators
\[
W_0^{1,p}(\Omega)\hookrightarrow
\begin{cases}
L^t(\Om),&\text{ for }t\in[1,p^{*}],\text{ if }p\in(1,N),\\
L^t(\Om),&\text{ for }t\in[1,\infty),\text{ if }p=N,\\
L^{\infty}(\Om),&\text{ if }p>N,
\end{cases}
\]
are continuous. Also, the above embeddings are compact, except for $t=p^*=\frac{Np}{N-p}$, if $p\in(1,N)$.
\end{Lemma}

Next, we state the algebraic inequality from \cite[Lemma $2.1$]{Dama}.

\begin{Lemma}\label{AI}
Let $1<p<\infty$. Then for any $a,b\in\mathbb{R}^N$, there exists a constant $C=C(p)>0$ such that
\begin{equation}\label{algineq}
\langle |a|^{p-2}a-|b|^{p-2}b, a-b \rangle\geq
C\frac{|a-b|^2}{(|a|+|b|)^{2-p}}.
\end{equation}
\end{Lemma}

The following result follows from \cite[Lemma A.2]{BrPr}.

\begin{Lemma}\label{BrPrapp}
Let $1<p<\infty$ and $g:\mathbb{R}\to\mathbb{R}$ be an increasing function. Then for every $a,b\in\mathbb{R}$, we have
\begin{equation}\label{BrPrine}
|a-b|^{p-2}(a-b)\big(g(a)-g(b)\big)\geq|G(a)-G(b)|^p,
\end{equation}
where
$$
G(t)=\int_{0}^{t}g'(\tau)^\frac{1}{p}\,d\tau,\quad t\in\mathbb{R}.
$$
\end{Lemma} 

For the following result, see \cite[Lemma $3.5$]{Caninononloc}.

\begin{Lemma}\label{Cnlemma}
Let $q>1$ and $\epsilon>0$. For $(x,y)\in\mathbb{R}^2$, let us set
$$
S_{\epsilon}^x:=\{x\geq\epsilon\}\cap\{y\geq 0\}\text{ and }S_{\epsilon}^y=\{y\geq\epsilon\}\cap\{x\geq 0\}.
$$
Then for every $(x,y)\in S_{\epsilon}^x\cup S_{\epsilon}^y$, we have 
\begin{equation}\label{cneqn}
|x-y|\leq\epsilon^{1-q}|x^q-y^q|.
\end{equation}
\end{Lemma}

Throughout the paper, we denote by
\begin{equation}\label{not}
\mathcal{A}\big(u(x,y)\big)=|u(x)-u(y)|^{p-2}\big(u(x)-u(y)\big)
\quad\text{and}\quad
d\mu=|x-y|^{-N-ps}\,dx\,dy.
\end{equation}
We define the notion of zero Dirichlet boundary condition as follows:
\begin{Definition}\label{bc}
We say that $u\leq 0$ on $\partial\Om$, if $u=0\text{ in }\mathbb{R}^N\setminus\Om$ and for every $\epsilon>0$, we have
$$
(u-\epsilon)^+\in W_0^{1,p}(\Om).
$$
We say that $u=0$ on $\partial\Om$, if $u$ is nonnegative and $u\leq 0$ on $\partial\Om$.
\end{Definition}

\begin{Remark}\label{bcrmk}
If $u\in W^{1,p}_{\mathrm{loc}}(\Om)$ is nonnegative in $\Om$, such that $u^\alpha\in W_0^{1,p}(\Om)$, for some $\alpha\geq 1$, then following the same proof of \cite[Theorem $2.10$]{G} we have $u=0$ on $\partial\Om$ according to the Definition \ref{bc}. 
\end{Remark}
Next, we define the notion of a weak solution for the problem $(\mathcal{S})$.

\begin{Definition}\label{defsolution}(Weak solution)
We say that a function $u\in W_{\mathrm{loc}}^{1,p}(\Omega)\cap L^{p-1}(\Om)$ is a weak solution of the problem $(\mathcal{S})$, if
$$
u>0\text{ in }\Om,\,u=0 \text{ on } \partial\Om \text{ in the sense of Definition \ref{bc}  and } \frac{f}{u^\delta}\in L^1_{\mathrm{loc}}(\Om),
$$
such that for every $\phi\in C_{c}^{1}(\Omega)$, we have
$$
\int_{\Omega}|\nabla u|^{p-2}\nabla u\nabla\phi\,dx+\int\limits_{\mathbb{R}^N}\int\limits_{\mathbb{R}^N}\mathcal{A}\big(u(x,y)\big)\big(\phi(x)-\phi(y)\big)\,d\mu=\int_{\Omega}\frac{f(x)}{u(x)^\delta}\phi(x)\,dx,\leqno{(\mathcal{W})}
$$
where $\mathcal{A}\big(u(x,y)\big)$ and $d\mu$ are defined by \eqref{not}.
\end{Definition}

\begin{Remark}\label{defrmk}
Note that Lemma \ref{defineq} and Lemma \ref{locnon1} ensures that Definition \ref{defsolution} is well stated. 
\end{Remark}

\subsection*{Statement of the main results:}
Now, we state our existence results as follows:
\begin{Theorem}\label{thm1}
Let $0<\delta<1$ and $1<p<\infty$. Assume that $f\in L^m(\Omega)\setminus\{0\}$ is nonnegative, where
\begin{equation}\label{m}
    m:=
\begin{cases}
    \big(\frac{p^{*}}{1-\delta}\big)',& \text{if } 1< p<N, \\
    t>1,& \text{if } p=N, \\
    1,& \text{if } p>N.
\end{cases}
\end{equation}
Then the problem $(\mathcal{S})$ admits a weak solution $u\in W_0^{1,p}(\Omega)$.
\end{Theorem}

\begin{Theorem}\label{thm2}
Let $\delta=1$ and $f\in L^1(\Om)\setminus\{0\}$ be nonnegative. Then for any $1<p<\infty$, the problem $(\mathcal{S})$ admits a weak solution $u\in W_0^{1,p}(\Om)$.
\end{Theorem}

\begin{Theorem}\label{thm3}
Let $\delta>1$ and $f\in L^1(\Om)\setminus\{0\}$ be nonnegative. Then for any $1<p<\infty$, the problem $(\mathcal{S})$ admits a weak solution $u\in W^{1,p}_{\mathrm{loc}}(\Om)$ such that $u^\frac{\delta+p-1}{p}\in W_0^{1,p}(\Om)$.
\end{Theorem}

Next, we have the following regularity result.
\begin{Theorem}\label{regthm}
Let $\delta>0$ and $f\in L^q(\Om)\setminus\{0\}$ be nonnegative such that $q>\frac{p^{*}}{p^{*}-p}$ if $1<p<N$, $q>\frac{l}{l-p}$ for some $l>p$ if $p=N$ and $q=1$ if $p>N$. Then, the weak solution given by Theorem \ref{thm1}-\ref{thm3} belong to $L^\infty(\Om)$.
\end{Theorem}

Now, we state our uniqueness result.
\begin{Theorem}\label{thm4}
Let $\delta>0$ and $f\in L^t(\Om)\setminus\{0\}$ be nonnegative such that $t=(p^*)'$ if $1<p<N$, $t>1$ if $p=N$ and $t=1$ if $p>N$. Then the problem $(\mathcal{S})$ admits at most one weak solution in $W^{1,p}_{\mathrm{loc}}(\Omega)\cap L^{p-1}(\Omega)$.
\end{Theorem}

As a consequence of Theorem \ref{thm4}, we obtain the following symmetry result:

\begin{Theorem}\label{sym}
Let $\delta>0$, $1<p<\infty$ and $f\in L^t(\Om)\setminus\{0\}$ be nonnegative such that $t=(p^*)'$ if $1<p<N$, $t>1$ if $p=N$ and $t=1$ if $p>N$. Assume that $\Om$ and $f$ be symmetric with respect to some hyperplane
$$
\mathcal{H}_{\lambda}^{\nu}:=\{x.\nu=\lambda\},\quad\lambda\in\mathbb{R},\quad\nu\in S^{N-1}.
$$
Then every weak solution $u$ in $W^{1,p}_{\mathrm{loc}}(\Omega)\cap L^{p-1}(\Omega)$ of the problem $(\mathcal{S})$ is symmetric with respect to the hyperplane $\mathcal{H}_{\lambda}^{\nu}$. In particular, if $\Om$ is a ball or an annulus centered at the origin and $f$ is radially symmetric, then $u$ is radially symmetric.
\end{Theorem}

Our main result associated with $\mu(\Omega)$ stated as follows:
\begin{Theorem}\label{thm5}
Let $0<\delta<1$ and $1<p<\infty$. Suppose that $f\in L^m(\Omega)\setminus\{0\}$ is nonnegative, where $m$ is given by \eqref{m} and $u_\delta\in W_0^{1,p}(\Om)$ is the weak solution of the problem $(\mathcal{S})$ given by Theorem \ref{thm1}. Then we have
\begin{enumerate} 
\item[(a)] (Extremal)
\begin{multline*}
\mu(\Omega):=\inf_{v\in W_{0}^{1,p}(\Omega)\setminus\{0\}}\left\{\int_{\Omega}|\nabla v|^p\,dx+\int\limits_{\mathbb{R}^N}\int\limits_{\mathbb{R}^N}\frac{|v(x)-v(y)|^p}{|x-y|^{N+ps}}\,dx\,dy:\int_{\Omega}|v|^{1-\delta}f\,dx=1\right\}\\
=\left(\int_{\Omega}|\nabla u_{\delta}|^p\,dx+\int\limits_{\mathbb{R}^N}\int\limits_{\mathbb{R}^N}\frac{|u_{\delta}(x)-u_{\delta}(y)|^p}{|x-y|^{N+ps}}\,dx\,dy\right)^\frac{1-\delta-p}{1-\delta}.
\end{multline*}
\item[(b)] (Mixed Sobolev Inequality) Moreover, for every $v\in W_{0}^{1,p}(\Omega)$, the following mixed Sobolev inequality
\begin{equation}\label{inequality2}
C\left(\int_{\Omega}|v|^{1-\delta}f\,dx\right)^\frac{p}{1-\delta}\leq\int_{\Omega}|\nabla v|^p\,dx+\int\limits_{\mathbb{R}^N}\int\limits_{\mathbb{R}^N}\frac{|v(x)-v(y)|^p}{|x-y|^{N+ps}}\,dx\,dy,
\end{equation}
holds, if and only if $$C\leq \mu(\Omega).$$
\item[(c)] (Simplicity) If for some $w\in W_0^{1,p}(\Om)\setminus\{0\}$, the equality
\begin{equation}\label{sim}
\mu(\Om)\left(\int_{\Omega}|w|^{1-\delta}f\,dx\right)^\frac{p}{1-\delta}=\int_{\Omega}|\nabla w|^p\,dx+\int\limits_{\mathbb{R}^N}\int\limits_{\mathbb{R}^N}\frac{|w(x)-w(y)|^p}{|x-y|^{N+ps}}\,dx\,dy,
\end{equation}
holds, then $w=ku_\delta$ for some constant $k$.
\end{enumerate}
\end{Theorem}
\begin{Corollary}\label{ansiowgtrmk}
From Theorem \ref{thm5}, we have
\begin{equation}\label{anisowgtrmk2}
\begin{split}
\mu(\Omega)=\int_{\Omega}|\nabla V_{\delta}|^p\,dx+\int\limits_{\mathbb{R}^N}\int\limits_{\mathbb{R}^N}\frac{|V_{\delta}(x)-V_{\delta}(y)|^p}{|x-y|^{N+ps}}\,dx\,dy,
\end{split}
\end{equation}
Moreover, $V_{\delta}\in S_\delta$ and satisfies the following singular problem
\begin{equation}\label{rmkeqn}
-\Delta_p V_{\delta}+(-\Delta_p)^s V_\delta=\mu(\Omega)f V_{\delta}^{-\delta}\text{ in }\Omega,\,V_{\delta}>0 \text{ in }\Omega.
\end{equation}
\end{Corollary}
We also obtain the following necessary and sufficient condition for the validity of the mixed Sobolev inequality.
\begin{Theorem}\label{nsthm}
Let $0<\delta<1<p<\infty$ and $f\in L^1(\Om)\setminus\{0\}$ be nonnegative. Then, for every $v\in W_0^{1,p}(\Om)$, the mixed Sobolev inequality \eqref{inequality2} holds, if and only if the problem $\mathcal{(S)}$ admits a weak solution in $W_0^{1,p}(\Om)$.
\end{Theorem}
Finally, we have the following results in the separate local and nonlocal cases.
\begin{Theorem}\label{nsthm1}
Let $0<\delta<1<p<\infty$ and $f\in L^1(\Om)\setminus\{0\}$ be nonnegative. Then, for every $v\in W_0^{1,p}(\Om)$, the Sobolev inequality
\begin{equation}\label{pine}
C\left(\int_{\Omega}|v|^{1-\delta}f\,dx\right)^\frac{p}{1-\delta}\leq\int_{\Omega}|\nabla v|^p\,dx,
\end{equation}
holds, if and only if the singular $p$-Laplace equation $(\mathcal{P})$ admits a weak solution in $W_0^{1,p}(\Om)$.
\end{Theorem}

\begin{Theorem}\label{nsthm2}
Let $0<\delta<1<p<\infty$ and $f\in L^1(\Om)\setminus\{0\}$ be nonnegative. Then, for every $v\in W_0^{s,p}(\Om)$, the Sobolev inequality
\begin{equation}\label{fine}
C\left(\int_{\Omega}|v|^{1-\delta}f\,dx\right)^\frac{p}{1-\delta}\leq\int\limits_{\mathbb{R}^N}\int\limits_{\mathbb{R}^N}\frac{|v(x)-v(y)|^p}{|x-y|^{N+ps}}\,dx\,dy,
\end{equation}
holds, if and only if the singular fractional $p$-Laplace equation $\mathcal{(F)}$ admits a weak solution in $W_0^{s,p}(\Om)$.
\end{Theorem}

\begin{Remark}\label{nsthmrmk}
The weak solutions for the problems $\mathcal{(P)}$ and $\mathcal{(F)}$ are defined analogous to the Definition \ref{defsolution} as in \cite[Definition $1.1$]{Caninononloc, Canino}.
\end{Remark}

As a consequence of Theorem \ref{nsthm}-\ref{nsthm1} and Lemma \ref{locnon1}, we have the following result connecting the problem $\mathcal{(P)}$ and $\mathcal{(S)}$.

\begin{Corollary}\label{nsthmcor}
Let $0<\delta<1<p<\infty$ and $f\in L^1(\Om)\setminus\{0\}$ be nonnegative. Then, the problem $\mathcal{(P)}$ admits a weak solution in $W_0^{1,p}(\Om)$, if and only if the problem $\mathcal{(S)}$ admits a weak solution in $W_0^{1,p}(\Om)$.
\end{Corollary}

\section{Preliminaries for the existence and regularity results}
Throughout this section, we assume the space $W_0^{1,p}(\Om)$ under the norm $\|\cdot\|$ given by \eqref{equinorm}. For $n\in\mathbb{N}$ and a nonnegative $f\in L^1(\Om)\setminus\{0\}$, let $f_n(x):=\min\{f(x),n\}$ and consider the following approximated problem
$$
-\Delta_p u+(-\Delta_p)^s u=\frac{f_n(x)}{\big(u^{+}+\frac{1}{n}\big)^{\delta}}\text{ in }\Omega,\,u=0\text{ in }\mathbb{R}^N\setminus\Om.\leqno{(\mathcal{A})}
$$
First, we obtain the following useful result for a general mixed nonsingular problem.
\begin{Lemma}\label{auxresult}
Let $g\in L^{\infty}(\Om)\setminus\{0\}$ be nonnegative. Then, there exists a unique solution $u\in W_0^{1,p}(\Omega)\cap L^{\infty}(\Om)$ of the problem
\begin{equation}\label{auxresulteqn}
\begin{split}
-\Delta_{p}u+(-\Delta_p)^s u=g\text{ in }\Om,\,u>0\text{ in }\Om,\,u=0\text{ in }\mathbb{R}^N\setminus\Om.
\end{split}
\end{equation}
Moreover, for every $\omega\Subset\Om$, there exists a constant $C(\omega)$ such that $u\geq C(\omega)>0$ in $\omega$.
\end{Lemma}
\begin{proof}
\textbf{Existence:} We define the energy functional $J:W_0^{1,p}(\Om)\to\mathbb{R}$ by
$$
J(v):=\frac{1}{p}\int_{\Omega}|\nabla v|^p\,dx+\frac{1}{p}\int\limits_{\mathbb{R}^N}\int\limits_{\mathbb{R}^N}\frac{|v(x)-v(y)|^p}{|x-y|^{N+ps}}\,dx\,dy-\int_{\Om}gv\,dx.
$$
\begin{itemize}
\item \textbf{Coercive:} We observe that $J$ is coercive. Indeed, by Lemma \ref{embedding} and using the fact that $g\in L^\infty(\Om)$, we have
\begin{equation}\label{coercive}
\begin{split}
J(v)&\geq\frac{\|v\|^p}{p}-C|\Om|^\frac{p-1}{p}\|g\|_{L^\infty(\Om)}\|v\|,
\end{split}
\end{equation}
where $C>0$ is the Sobolev constant. Since $p>1$, from \eqref{coercive}, we get $J$ is coercive.

\item \textbf{Weak lower semicontinuity:} Firstly, we observe that $J$ is convex. To this end, we write $J=J_1+J_2+J_3$, where for $i=1,2,3$ we define $J_i:W_0^{1,p}(\Om)\to\mathbb{R}$ as follows: 
$$
J_1(v):=\frac{1}{p}\int_{\Om}|\nabla v|^p\,dx,
$$

$$
J_2(v):=\frac{1}{p}\int\limits_{\mathbb{R}^N}\int\limits_{\mathbb{R}^N}\frac{|v(x)-v(y)|^p}{|x-y|^{N+ps}}\,dx\,dy,
$$
and

$$
J_3(v):=-\int_{\Om}gv\,dx.
$$
Since $p>1$, using the convexity of $|\cdot|^p$, we obtain that both the functional $J_1$ and $J_2$ are convex. It directly follows that $J_3$ is the linear functional. As a consequence, the functional $J$ become convex and moreover $J$ is a $C^1$ functional. Thus $J$ is weakly lower semicontinuous.
\end{itemize}
Thus the coercive, weak lower semicontinuty and $C^1$ property of $J$ asserts that $J$ has a minimizer, say $u\in W_0^{1,p}(\Omega)$ which solves the following problem
\begin{equation}\label{auxeqn}
-\Delta_p u+(-\Delta_p)^s u=g\text{ in }\Om.
\end{equation}
\textbf{Uniqueness:} Let $u_1,u_2\in W_0^{1,p}(\Om)$ solves the problem \eqref{auxeqn}. Thus, for every $\phi\in W_0^{1,p}(\Om)$, we have
\begin{equation}\label{auxeqn1}
\begin{split}
\int_{\Omega}|\nabla u_1|^{p-2}\nabla u_1\nabla\phi\,dx+\int\limits_{\mathbb{R}^N}\int\limits_{\mathbb{R}^N}\mathcal{A}\big(u_1(x,y)\big)\big(\phi(x)-\phi(y)\big)\,d\mu=\int_{\Omega}g\phi\,dx,
\end{split}
\end{equation}
and
\begin{equation}\label{auxeqn2}
\begin{split}
\int_{\Omega}|\nabla u_2|^{p-2}\nabla u_2\nabla\phi\,dx+\int\limits_{\mathbb{R}^N}\int\limits_{\mathbb{R}^N}\mathcal{A}\big(u_2(x,y)\big)\big(\phi(x)-\phi(y)\big)\,d\mu=\int_{\Omega}g\phi\,dx.
\end{split}
\end{equation}
Choosing $\phi=u_1-u_2$ and then subtracting \eqref{auxeqn1} and \eqref{auxeqn2}, we arrive at
\begin{equation}\label{auxeqn3}
\begin{split}
&\int_{\Omega}\left(|\nabla u_1|^{p-2}\nabla u_1-|\nabla u_2|^{p-2}\nabla u_2\right)\nabla(u_1-u_2)\,dx\\
&+\int\limits_{\mathbb{R}^N}\int\limits_{\mathbb{R}^N}\left(\mathcal{A}\big(u_1(x,y)\big)-\mathcal{A}\big(u_2(x,y)\big)\right)\big((u_1-u_2)(x)-(u_1-u_2)(y)\big)\,d\mu=0.
\end{split}
\end{equation}
Applying Lemma \ref{AI}, both the terms in the above estimate become nonnegative and hence we obtain from \eqref{auxeqn3},
\begin{equation}\label{auxeqn4}
\begin{split}
&\int_{\Omega}\left(|\nabla u_1|^{p-2}\nabla u_1-|\nabla u_2|^{p-2}\nabla u_2\right)\nabla(u_1-u_2)\,dx=0.
\end{split}
\end{equation}
Again, applying Lemma \ref{AI} in \eqref{auxeqn4}, we obtain $u_1=u_2$ in $\Om$. Therefore the solution of \eqref{auxeqn} is unique.

\noindent
\textbf{Boundedness:} In order to prove the boundedness, let
$$
A(k):=\{x\in\Omega:u(x)\geq k\}\,\,\text{for any}\,\, k\geq 1.
$$
Choosing $\phi_k:=(u-k)^+=\max\{u-k,0\}$ as a test function in \eqref{auxeqn}, we have
\begin{equation}\label{tstbd1}
\begin{split}
\int_{\Omega}|\nabla u|^{p-2}\nabla u\nabla\phi_k\,dx+\int\limits_{\mathbb{R}^N}\int\limits_{\mathbb{R}^N}\mathcal{A}\big(u(x,y)\big)\big(\phi_k(x)-\phi_k(y)\big)\,d\mu=\int_{\Omega}g\phi_k\,dx.
\end{split}
\end{equation}
First we estimate the nonlocal integral in \eqref{tstbd1}. To this end, we observe that
\begin{equation}\label{nonlosign}
\begin{split}
&\mathcal{A}\big(u(x,y)\big)\big(\phi_k(x)-\phi_k(y)\big)\\
&=|u(x)-u(y)|^{p-2}\big(u(x)-u(y)\big)\big((u(x)-k)^+-(u(y)-k)^+\big)\\&=
\begin{cases}
|u(x)-u(y)|^p,\text{ if }u(x)>k,\,u(y)>k,\\
(u(x)-u(y))^{p-1}\big(u(x)-k\big),\text{ if }u(x)>k\geq u(y),\\
\big(u(y)-u(x)\big)^{p-1}\big(u(y)-k\big),\text{ if }u(y)>k\geq u(x),\\
0,\text{ if }u(x)\leq k,\,u(y)\leq k.
\end{cases}\\
&\geq 0.
\end{split}
\end{equation}
Therefore, using \eqref{nonlosign} in \eqref{tstbd1} along with the continuity of the mapping $W_0^{1,p}(\Om)\hookrightarrow L^l(\Omega)$ for some $l>p$, from Lemma \ref{embedding}
\begin{multline}
\label{tstbd2}
\int_{\Omega}|\nabla\phi_k|^p\,dx=\int_{\Omega}|\nabla u|^{p-2}\nabla u\nabla\phi_k\,dx\leq\int_{\Omega}g\phi_k\,dx
\leq \|g\|_{L^{\infty}(\Om)}\int_{A(k)}(u-k)\,dx\\
\leq C_0\|g\|_{L^{\infty}(\Om)}|A(k)|^\frac{l-1}{l}\left(\int_{\Omega}|\nabla\phi_k|^p\,dx\right)^\frac{1}{p},
\end{multline}
where $C_0$ is the Sobolev constant. Hence, we have
\begin{equation}\label{new}
    \int_{\Omega}|\nabla\phi_k|^p\,dx \leq C|A(k)|^\frac{p(l-1)}{l(p-1)},
\end{equation}
for some positive constant $C=C(C_0,\|g\|_{L^{\infty}(\Om)})$. Now choose $h$ such that $1\leq k<h$. Then $u(x)-k \geq (h-k)$ on $A(h)$ and $A(h)\subset A(k)$. Noting this fact along with \eqref{new}, we get
\begin{multline*}
(h-k)^p|A(h)|^\frac{p}{l}\leq\left(\int_{A(h)}\big(u(x)-k\big)^l\,dx\right)^\frac{p}{l}\leq\left(\int_{A(k)}\big(u(x)-k\big)^l~dx\right)^\frac{p}{l}\\
\leq C_0\int_{\Omega}|\nabla\phi_k|^p~dx \leq C|A(k)|^\frac{p(l-1)}{l(p-1)},
\end{multline*}
for some positive constant $C=C(C_0,\|g\|_{L^{\infty}(\Om)})$. Therefore, we have
$$
|A(h)|\leq\frac{C}{(h-k)^l}|A(k)|^\frac{l-1}{p-1}.
$$
We observe that $\frac{l-1}{p-1}>1$. Thus using \cite[Lemma B.1]{Stam} we obtain
$$
\|u\|_{L^\infty(\Omega)}\leq C,
$$
for some positive constant $C=C(C_0,\|g\|_{L^{\infty}(\Om)})$. Hence, $u\in L^\infty(\Om)$.\\
\textbf{Positivity:} Choosing $u_-:=\min\{u,0\}$ as a test function in \eqref{auxeqn} and using $g\geq 0$, we have
\begin{equation}\label{posi}
\begin{split}
\int_{\Omega}|\nabla u_-|^p\,dx+\int\limits_{\mathbb{R}^N}\int\limits_{\mathbb{R}^N}\mathcal{A}\big(u(x,y)\big)\big(u_-(x)-u_-(y)\big)\,d\mu&=\int_{\Omega}gu_-\,dx\leq 0,
\end{split}
\end{equation}
where $\mathcal{A}\big(u(x,y)\big)=|u(x)-u(y)|^{p-2}\big(u(x)-u(y)\big)$.
We write,
$$
\mathbb{R}^N\times\mathbb{R}^N=\cup_{i=1}^{4}S_i,
$$
where
\begin{equation*}
\begin{split}
S_1&=\big\{(x,y)\in\mathbb{R}^N\times\mathbb{R}^N:u(x)\geq 0, u(y)\geq 0\big\},\\
S_2&=\big\{(x,y)\in\mathbb{R}^N\times\mathbb{R}^N:u(x)\geq 0, u(y)<0\big\},\\
S_3&=\big\{(x,y)\in\mathbb{R}^N\times\mathbb{R}^N:u(x)<0, u(y)\geq 0\big\},\\
S_4&=\big\{(x,y)\in\mathbb{R}^N\times\mathbb{R}^N:u(x)<0, u(y)<0\big\}.
\end{split}
\end{equation*}
For $(x,y)\in S_1$, we have $u_-(x)=u_-(y)=0$. If $(x,y)\in S_2$, then $u_-(x)=0$, $u_-(y)=u(y)$, which gives $u(x)-u(y)\geq u_-(x)-u_-(y)>0$. When $(x,y)\in S_3$, we have $u_-(x)=u(x),u_-(y)=0$. Hence, $u(x)-u(y)\leq u_-(x)-u_-(y)<0$. In case of $(x,y)\in S_4$, we have $u_-(x)=u(x)$ and $u_-(y)=u(y)$.\\
Therefore, for any $S_i$ with $i=1,\ldots,4$, if $(x,y)\in S_i$, then we have
\begin{equation}\label{pos}
|u(x)-u(y)|^{p-2}\big(u(x)-u(y)\big)\big(u_-(x)-u_-(y)\big)\geq 0.
\end{equation}
Using \eqref{pos} in \eqref{posi} we obtain
\begin{equation*}
\int_{\Omega}|\nabla u_-|^p\,dx=0.
\end{equation*}
Hence, $u\geq 0$ in $\Om$. Since $u=0$ in $\mathbb{R}^N\setminus\Om$, the function $u$ is nonnegative in $\mathbb{R}^N$ and from the above step, $u\in L^\infty(\mathbb{R}^N)$. As $g\not\equiv 0$, we have $u\not\equiv 0$ in $\Omega$. Thus, using \cite[Theorem $8.2$]{GK} for every $\omega\Subset\Omega$, there exists a constant $C(\omega)>0$ such that $u\geq C(\omega)>0$ in $\Omega$. Hence $u>0$ in $\Omega$.
\end{proof}

Next, we obtain the following existence and further qualitative properties of solutions for the problem $\mathcal{(A)}$.
\begin{Lemma}\label{approx}
For every $n\in\mathbb{N}$, there exists a unique positive solution $u_n\in W_{0}^{1,p}(\Omega)\cap L^{\infty}(\Omega)$ of $(\mathcal{A})$. Moreover, $u_{n+1}\geq u_n$ in $\Omega$, for every $n\in\mathbb{N}$. Furthermore, for every $n\in\mathbb{N}$ and every $\omega\Subset\Omega$, there exists a constant $C(\omega)>0$ (independent of $n$) such that $u_n\geq C(\omega)>0$ in $\omega$.
\end{Lemma}
\begin{proof}
\textbf{Existence:} Let $n\in\mathbb{N}$ be fixed. Then, by Lemma \ref{auxresult}, for every $h\in L^p(\Omega)$, there exists a unique $v\in W_0^{1,p}(\Omega)\cap L^\infty(\Om)$ such that
\begin{equation}\label{approxfixed}
\begin{split}
-\Delta_p v+(-\Delta)_p^{s}v=\frac{f_n}{\big(h^{+}+\frac{1}{n}\big)^\delta}\text{ in }\Omega,\,v>0\text{ in }\Omega,\,v=0\text{ in }\mathbb{R}^N\setminus\Om.
\end{split}
\end{equation}
Therefore we can define the operator $T:L^p(\Omega)\to L^p(\Omega)$ by $T(h)=v$, where $v$ solves \eqref{approxfixed}. Choosing $v$ as a test function in \eqref{approxfixed}, using Lemma \ref{embedding}, we obtain
\begin{align*}
\int_{\Omega}|\nabla v|^p\,dx\leq\int_{\Omega}n^{\delta+1}v\,dx\leq C n^{\delta+1}|\Omega|^\frac{p-1}{p}\Big(\int_{\Omega}|\nabla v|^p\,dx\Big)^\frac{1}{p},
\end{align*}
where $C>0$ is the Sobolev constant. This gives
$$
\Big(\int_{\Omega}|\nabla v|^p\,dx\Big)^\frac{1}{p}\leq C^\frac{1}{p-1} n^\frac{\delta+1}{p-1}|\Omega|^\frac{1}{p}.
$$
Keeping in mind Lemma \ref{embedding} and the above estimate, arguing exactly as in the proof of \cite[Proposition $2.3$]{Caninononloc}, it follows that the mapping $T$ is continuous and compact. Hence, by the Schauder fixed point theorem, there exists a fixed point of $T,$ say $u_n\in W_0^{1,p}(\Omega)\cap L^\infty(\Om)$. Thus, $u_n$ solves the problem $\mathcal{(A)}$. Moreover, by Lemma \ref{auxresult}, we have $u_n>0$ in $\Omega$ such that, for every $\omega\Subset$ there exists a positive constant $C(\omega)$ satisfying $u_n\geq C(\omega)>0$ in $\omega$.

\noindent
\textbf{Monotonicity and uniqueness:} Since $u_n$ and $u_{n+1}$ are positive solutions of the problem $(\mathcal{A})$, for every $\phi\in W_0^{1,p}(\Omega)$, we have
\begin{equation}\label{monpre1}
\int_{\Omega}|\nabla u_n|^{p-2}\nabla u_n\nabla\phi\,dx+\int\limits_{\mathbb{R}^N}\int\limits_{\mathbb{R}^N}\mathcal{A}\big(u_n(x,y)\big)\big(\phi(x)-\phi(y)\big)\,d\mu=\int_{\Omega}\frac{f_n(x)}{\big(u_n+\frac{1}{n}\big)^\delta}\phi\,dx,
\end{equation}

\begin{equation}\label{monpre2}
\int_{\Omega}|\nabla u_{n+1}|^{p-2}\nabla u_{n+1}\nabla\phi\,dx+\int\limits_{\mathbb{R}^N}\int\limits_{\mathbb{R}^N}\mathcal{A}\big(u_{n+1}(x,y)\big)\big(\phi(x)-\phi(y)\big)\,d\mu=\int_{\Omega}\frac{f_{n+1}(x)}{\big(u_{n+1}+\frac{1}{n+1}\big)^\delta}\phi\,dx.
\end{equation}
Choosing $\phi=(u_{n}-u_{n+1})^+$ as a test function in \eqref{monpre1} and \eqref{monpre2}, we have
\begin{multline}\label{mon1}
\int_{\Omega}|\nabla u_n|^{p-2}\nabla u_n\nabla(u_n-u_{n+1})^+\,dx\\
\quad+\int\limits_{\mathbb{R}^N}\int\limits_{\mathbb{R}^N}\mathcal{A}\big(u_n(x,y)\big)\Big(\big(u_n-u_{n+1}\big)^+(x)-\big(u_n-u_{n+1}\big)^+(y)\Big)\,d\mu\\
\quad\quad=\int_{\Omega}\frac{f_n(x)}{\big(u_n+\frac{1}{n}\big)^\delta}\big(u_n-u_{n+1}\big)^+(x)\,dx,
\end{multline}
and
\begin{multline}\label{mon2}
\int_{\Omega}|\nabla u_{n+1}|^{p-2}\nabla u_{n+1}\nabla(u_n-u_{n+1})^+\,dx\\
\quad+\int\limits_{\mathbb{R}^N}\int\limits_{\mathbb{R}^N}\mathcal{A}\big(u_{n+1}(x,y)\big)\Big(\big(u_n-u_{n+1}\big)^+(x)-\big(u_n-u_{n+1}\big)^+(y)\Big)\,d\mu\\
\quad\quad=\int_{\Omega}\frac{f_{n+1}(x)}{\big(u_{n+1}+\frac{1}{n+1}\big)^\delta}\big(u_n-u_{n+1}\big)^+(x)\,dx,
\end{multline}
respectively. Noting the fact $f_{n}(x) \leq f_{n+1}(x)$ for $x\in\Om$, we have
\begin{equation}\label{monrhs}
\int_{\Omega}\left\{\frac{f_{n}(x)}{\big(u_{n}+\frac{1}{n}\big)^\delta}-\frac{f_{n+1}(x)}{\big(u_{n+1}+\frac{1}{n+1}\big)^\delta}\right\}\big(u_n-u_{n+1}\big)^+(x)\,dx\leq 0.
\end{equation}
Subtracting \eqref{mon1} with \eqref{mon2} and using the fact \eqref{monrhs}, we get
\begin{equation}\label{monest}
\begin{split}
&\int_{\Omega}\big(|\nabla u_{n}|^{p-2}\nabla u_{n}-|\nabla u_{n+1}|^{p-2}\nabla u_{n+1}\big)\nabla(u_n-u_{n+1})^+\,dx\\
&\quad+\int\limits_{\mathbb{R}^N}\int\limits_{\mathbb{R}^N}\Big(\mathcal{A}\big(u_n(x,y)\big)-\mathcal{A}\big(u_{n+1}(x,y)\big)\Big)\Big(\big(u_n-u_{n+1}\big)^+(x)-\big(u_n-u_{n+1}\big)^+(y)\Big)\,d\mu\leq 0.
\end{split}
\end{equation}
Following the same arguments from the proof of \cite[Lemma $9$]{Ling}, we obtain
\begin{equation}\label{monnon}
\int\limits_{\mathbb{R}^N}\int\limits_{\mathbb{R}^N}\Big(\mathcal{A}\big(u_n(x,y)\big)-\mathcal{A}\big(u_{n+1}(x,y)\big)\Big)\Big(\big(u_n-u_{n+1}\big)^+(x)-\big(u_n-u_{n+1}\big)^+(y)\Big)\,d\mu\geq 0.
\end{equation}
Hence, applying \eqref{monnon} in \eqref{monest} we obtain
\begin{equation}\label{monfinal}
\int_{\Omega}\big(|\nabla u_{n}|^{p-2}\nabla u_{n}-|\nabla u_{n+1}|^{p-2}\nabla u_{n+1}\big)\nabla(u_n-u_{n+1})^+\,dx\leq 0.
\end{equation}
Then, using Lemma \ref{AI} we obtain $u_{n+1}\geq u_n$ in $\Omega.$ Uniqueness follows similarly.

\noindent
\textbf{Uniform Positivity:} By Lemma \ref{auxresult} for every $\omega\Subset\Omega$, there exists a constant $C(\omega)>0$ such that $u_1\geq C(\omega)>0$ in $\omega$. Again, by the monotonicity we have $u_n\geq u_1$ in $\Omega$ for every $n\in\mathbb{N}$. Hence, for every $\omega\Subset\Omega$ and every $n\in\mathbb{N}$,
$$
u_n(x)\geq C(\omega)>0,\text{ for }x\in\omega,
$$
where $C(\omega)>0$ is a constant independent of $n$.
\end{proof}
\begin{Remark}\label{rmkapprox}
By Lemma \ref{approx}, since $\{u_n\}$ is monotone, we can define the pointwise limit of $u_n$ in $\Om$, say by $u$. Hence, $u\geq u_n$ in $\mathbb{R}^N$, for every $n\in\mathbb{N}$. Below, we prove that $u$ is our required solution for the problem $(\mathcal{S})$.
\end{Remark}

Next, we obtain some boundedness estimates (Lemma \ref{unibddless}-\ref{unibddreg}) for the sequence of positive solutions $\left\{u_n\right\}$ of the problem $\mathcal{(A)}$ given by Lemma \ref{approx}. These estimates are important to deduce the existence and regularity results.

\begin{Lemma}\label{unibddless}
Suppose $0<\delta<1$ and $f\in L^m(\Om)\setminus\{0\}$ be nonnegative such that $m=\big(\frac{p^{*}}{1-\delta}\big)'$ if $1<p<N$, $m>1$ if $p=N$ and $m=1$ if $p>N$. Then, the sequence $\{u_n\}$, where $u_n$, $n\in\mathbb{N}$ are solutions of the approximate problem $(\mathcal{A})$, is uniformly bounded in $W_0^{1,p}(\Om)$.
\end{Lemma}

\begin{proof}
Because $u_n$ are solutions of  $(\mathcal{A})$, choosing $u_n$ as a test function in $\mathcal{(A)}$, we get
$$
\|u_n\|^p\leq \int_{\Om}\frac{f_n(x)u_n}{\big(u_n+\frac{1}{n}\big)^\delta}\,dx.
$$
Let $1<p<N$, then noting that $f\in L^m(\Om)$ and $(1-\delta)m'=p^*$, along with Lemma \ref{embedding}
\begin{multline*}
\|u_n\|^p\leq \int_{\Om}f u_n^{1-\delta}\,dx\leq \|f\|_{L^m(\Om)}\left(\int_{\Om}u_n^{(1-\delta)m'}\right)^\frac{1}{m'}\\
=\|f\|_{L^m(\Om)}\left(\int_{\Om}u_n^{p^*}\,dx\right)^\frac{1-\delta}{p^*}
\leq C\|f\|_{L^m(\Om)}\|u_n\|^{1-\delta},
\end{multline*}
for some constant $C>0$, independent of $n$. Therefore, we have
$$
\|u_n\|\leq C.
$$
Hence, the sequence $\{u_n\}$ is uniformly bounded in $W_0^{1,p}(\Om)$. For $p\geq N$, the result follows similarly.
\end{proof}
In fact, the uniform boundedness of $\{u_n\}$ follows for any integrable function if the mixed Sobolev inequality \eqref{inequality2} holds. More precisely, we have the following result.

\begin{Lemma}\label{nsbdd}
Suppose $0<\delta<1$ and $f\in L^1(\Om)\setminus\{0\}$ be nonnegative. Then for any $1<p<\infty$, the sequence $\{u_n\}$, where $u_n$, $n\in\mathbb{N}$ are solutions of the approximate problem $(\mathcal{A})$, is uniformly bounded in $W_0^{1,p}(\Om)$, provided \eqref{inequality2} holds.
\end{Lemma}

\begin{proof}
Choosing $u_n$ as a test function in $(\mathcal{A})$ and using the inequality \eqref{inequality2} and Lemma \ref{locnon1}, we have 
\begin{equation*}
\|u_n\|^p\leq \int_{\Om}\frac{f_n(x)u_n}{\big(u_n+\frac{1}{n}\big)^\delta}\,dx
\leq\int_{\Om}u_n^{1-\delta}f\,dx
\leq C\|u_n\|^{1-\delta},
\end{equation*}
for some constant $C>0$ (independent of $n$). Hence, the result follows.
\end{proof}

\begin{Lemma}\label{unibdequal}
Suppose $\delta=1$ and $f\in L^1(\Om)\setminus\{0\}$ is nonnegative in $\Om$. Then, for any $1<p<\infty$, the sequence $\{u_n\}$, where $u_n$, $n\in\mathbb{N}$ are solutions of the approximate problem $(\mathcal{A})$, is uniformly bounded in $W_0^{1,p}(\Om)$.
\end{Lemma}

\begin{proof}
Choosing $\phi=u_n$ as a test function in $(\mathcal{A})$ we obtain
\begin{equation}\label{}
\begin{split}
\|u_n\|^p&\leq\int_{\Om}\frac{f_n}{\big(u_n+\frac{1}{n}\big)}u_n\,dx\\
&\leq\|f\|_{L^1(\Om)}.
\end{split}
\end{equation}
Thus, $\{u_n\}$ is uniformly bounded in $W_0^{1,p}(\Om)$.
\end{proof}

\begin{Lemma}\label{unibddgrt}
Suppose $\delta>1$ and $f\in L^1(\Om)\setminus\{0\}$ is nonnegative in $\Om$. Then, for any $1<p<\infty$, the sequences $\left\{u_n^\frac{\delta+p-1}{p}\right\}$ and $\{u_n\}$, where $u_n$, $n\in\mathbb{N}$ are solutions of the approximate problem $(\mathcal{A})$, is uniformly bounded in $W^{1,p}_0(\Om)$ and $W_{\mathrm{loc}}^{1,p}(\Om)$ respectively.
\end{Lemma}

\begin{proof}
\textbf{Uniform boundedness in $W_0^{1,p}(\Om)$:} By Lemma \ref{approx} for every $n\in\mathbb{N}$, we have $u_n\in L^\infty(\Om)$. Therefore, since $\delta>1$, we choose $u_n^{\delta}\in W_0^{1,p}(\Om)$ as a test function in $(\mathcal{A})$ to obtain
\begin{equation}\label{deltabig1}
\begin{split}
\int_{\Om}|\nabla u_n|^{p-2}\nabla u_n\nabla u_n^{\delta}\,dx+\int\limits_{\mathbb{R}^N}\int\limits_{\mathbb{R}^N}\mathcal{A}\big(u_n(x,y)\big)\big((u_n(x)^\delta-u_n(y)^\delta\big)\,d\mu&=\int_{\Omega}\frac{f(x)}{\big(u_n+\frac{1}{n}\big)^\delta}u_n^\delta\,dx.
\end{split}
\end{equation}
Using Lemma \ref{BrPrapp}, for almost every $x,y\in\mathbb{R}^N$, we have
\begin{equation}\label{BrPrineq}
\mathcal{A}\big(u_n(x,y)\big)\big(u_n(x)^\delta-u_n(y)^\delta\big)=|u_n(x)-u_n(y)|^{p-2}\big(u_n(x)-u_n(y)\big)\big(u_n(x)^\delta-u_n(y)^\delta\big)\geq 0.
\end{equation}
Therefore, from \eqref{BrPrineq} and \eqref{deltabig1}, we obtain
\begin{align*}
\int_{\Om}\delta\Big(\frac{p}{\delta+p-1}\Big)^p\left|\nabla u_n^\frac{\delta+p-1}{p}\right|^p\,dx=\int_{\Om}|\nabla u_n|^{p-2}\nabla u_n\nabla u_n^{\delta}\,dx\leq\int_{\Omega}\frac{f(x)}{\big(u_n+\frac{1}{n}\big)^\delta}u_n^\delta\,dx\leq\|f\|_{L^1(\Om)},
\end{align*}
which gives
\begin{equation}\label{unibddpower}
\left\|u_n^\frac{\delta+p-1}{p}\right\|\leq C,
\end{equation}
for some constant $C>0$ independent of $n$. Thus, the sequence $\left\{u_n^\frac{\delta+p-1}{p}\right\}$ is uniformly bounded in $W_0^{1,p}(\Om)$.

\noindent
\textbf{Uniform boundedness in $W^{1,p}_{\mathrm{loc}}(\Om)$:}  Let $\omega\Subset\Om$, then by Lemma \ref{approx}, $u_n\geq C(\omega)>0$ in $\omega$ for every $n\in\mathbb{N}$, where $C(\omega)>0$ is a constant independent of $n$. Since, $\delta>1$, then
\begin{equation}\label{1}
\int_{\omega}|\nabla u|^p\,dx=\Big(\frac{p}{\delta+p-1}\Big)^p\int_{\omega}u_n^{1-\delta}\left|\nabla u_n^\frac{\delta+p-1}{p}\right|^p\,dx
\leq\Big(\frac{p}{\delta+p-1}\Big)^p C(\omega)^{1-\delta} C^p, 
\end{equation}
where the constant $C$ is given by \eqref{unibddpower}. Therefore, the sequence $\left\{\nabla u_n\right\}$ is uniformly bounded in $L^p_{\mathrm{loc}}(\Om)$. Again using \eqref{unibddpower} and H\"older's inequality, we obtain
\begin{equation}\label{2}
\int_{\Om}u_n^p\,dx\leq|\Om|^\frac{\delta-1}{\delta+p-1}\left(\int_{\Om}u_n^{\delta+p-1}\,dx\right)^\frac{p}{\delta+p-1}\leq C,
\end{equation}
for some constant $C>0$, independent of $n$. Hence, from \eqref{1} and \eqref{2}, it turns out that $\left\{u_n\right\}$ is uniformly bounded in $W^{1,p}_{\mathrm{loc}}(\Om)$.
\end{proof}

The following a priori estimate is crucial to prove our regularity result.
\begin{Lemma}\label{unibddreg}
Let $\delta>0$ and $f\in L^q(\Omega)\setminus\{0\}$ be nonnegative such that $q>\frac{p^{*}}{p^{*}-p}$ if $1<p<N$, $q>\frac{l}{l-p}$ for some $l>p$ if $p=N$ and $q=1$ if $p>N$. Then $\|u_n\|_{L^\infty(\Omega)}\leq C$, for some positive constant $C$ independent of $n$.
\end{Lemma}
\begin{proof}
We prove the result only for the case $1<p<N$, since the proof for $p\geq N$ is analogous. Let $k\geq 1$ and define $A(k)=\{x\in\Om: u_n(x)\geq k\}$. Choosing $\phi_k(x)=(u_n-k)^{+}\in W_0^{1,p}(\Om)$ as a test function in $(\mathcal{A})$, using H\"older's inequality with the exponents $p^{*'}, p^*$ and then by Young's inequality with exponents $p$ and $p'$ along with \eqref{nonlosign} and Lemma \ref{embedding} we obtain
\begin{multline}
\label{unibdd}
\int_{\Om}|\nabla\phi_k|^p\,dx\leq\int_{\Om}\frac{f_n}{\big(u_n+\frac{1}{n}\big)^\delta}\phi_k\,dx\leq\int_{A(k)}f(x)\phi_k\,dx\\
\leq\left(\int_{A(k)}f^{p^{*'}}\,dx\right)^\frac{1}{p^{*'}}\left(\int_{\Om}\phi_k^{p^*}\,dx\right)^\frac{1}{p^*}
\leq C\left(\int_{A(k)}f^{p^{*'}}\,dx\right)^\frac{1}{p^{*'}}\left(\int_{\Om}|\nabla\phi_k|^{p}\,dx\right)^\frac{1}{p}\\
\leq \epsilon\int_{\Om}|\nabla\phi_k|^p\,dx+C(\epsilon)\left(\int_{A(k)}f^{p^{*'}}\,dx\right)^\frac{p'}{p^{*'}}.
\end{multline}
Here, $C$ is the Sobolev constant and  $C(\epsilon)>0$ is some constant depending on $\epsilon\in(0,1)$. Note that $q>\frac{p^*}{p^*-p}$ gives $q>p^{*'}$. Therefore, fixing $\epsilon\in(0,1)$ and again using H\"older's inequality with exponents $\frac{q}{p^{*'}}$ and $\big(\frac{q}{p^{*'}}\big)'$, we obtain
\begin{align*}
\int_{\Om}|\nabla\phi_k|^p\,dx\leq C\left(\int_{A(k)}f^{p^{*'}}\,dx\right)^\frac{p'}{p^{*'}}\leq C\left(\int_{A(k)}f^q\,dx\right)^\frac{p'}{q}|A(k)|^{\frac{p'}{p^{*'}}\frac{1}{\big(\frac{q}{p^{*'}}\big)'}}.
\end{align*}
Let $h>0$ be such that $1\leq k<h$. Then, $A(h)\subset A(k)$ and for any $x\in A(h)$, we have $u_n(x)\geq h$. So, $u_n(x)-k\geq h-k$ in $A(h)$. 
Noting these facts, for some constant $C>0$ (independent of $n$), we have
\begin{multline*}
(h-k)^p|A(h)|^\frac{p}{p^*}\leq\left(\int_{A(h)}(u_n-k)^{p^*}\,dx\right)^\frac{p}{p^*}\leq\left(\int_{A(k)}(u_n-k)^{p^*}\,dx\right)^\frac{p}{p^*}\\
\leq C\int_{\Om}|\nabla\phi_k|^p\,dx\leq C\|f\|_{L^q(\Om)}^{p'}|A(k)|^{\frac{p'}{p^{*'}}\frac{1}{\big(\frac{q}{p^{*'}}\big)'}}.
\end{multline*}
Thus, for some constant $C>0$ (independent of $n$), we have
$$
|A(h)|\leq C\frac{\|f\|_{L^q(\Om)}^\frac{p^*}{p-1}}{(h-k)^{p^*}}|A(k)|^{\alpha},\,\,\text{where}\,\,\alpha={\frac{p^{*}p'}{pp^{*'}}\frac{1}{\big(\frac{q}{p^{*'}}\big)'}}.
$$
Due to the assumption, $q>\frac{p^*}{p^*-p}$, we have $\alpha>1$. Hence, by \cite[Lemma B.$1$]{Stam}, we have
$$
\|u_n\|_{L^\infty(\Om)}\leq C,
$$
for some positive constant $C>0$, independent of $n$.
\end{proof}

\section{Preliminaries for the uniqueness results}
Throughout this section, we assume that $\delta>0$ and $f\in L^t(\Om)\setminus\{0\}$ is nonnegative such that $t=(p^*)'$ if $1<p<N$, $t>1$ if $p=N$ and $t=1$ if $p>N$. Here, we obtain a comparison principle (Lemma \ref{cp}), that is crucial to obtain the uniqueness result. Firstly, we define the notion of weak subsolution and weak supersolution for the problem $(\mathcal{S})$.
\begin{Definition}\label{supsol} (Weak supersolution)
We say that $u\in W_{\mathrm{loc}}^{1,p}(\Omega)\cap L^{p-1}(\Om)$ is a weak supersolution of the problem $(\mathcal{S})$, if 
$$
u>0\text{ in }\Om,\,u=0 \text{ in }\mathbb{R}^N\setminus\Om \text{ and } \frac{f}{u^\delta}\in L^1_{\mathrm{loc}}(\Om),
$$
such that for every $\phi\in C_c^{1}(\Om)$, we have
\begin{equation}\label{supsolution} 
\int_{\Omega}|\nabla u|^{p-2}\nabla u\nabla\phi\,dx+\int\limits_{\mathbb{R}^N}\int\limits_{\mathbb{R}^N}\mathcal{A}\big(u(x,y)\big){\big(\phi(x)-\phi(y)\big)}\,d\mu\geq\int_{\Omega}\frac{f(x)}{u(x)^\delta}\phi(x)\,dx.
\end{equation}
\end{Definition}

\begin{Definition}\label{subsol} (Weak subsolution)
We say that $u\in W_{\mathrm{loc}}^{1,p}(\Omega)\cap L^{p-1}(\Om)$ is a weak subsolution of the problem $(\mathcal{S})$, if 
$$
u>0\text{ in }\Om,\,u=0 \text{ in }\mathbb{R}^N\setminus\Om\text{ and }\frac{f}{u^\delta}\in L^1_{\mathrm{loc}}(\Om),
$$
such that for every $\phi\in C_c^{1}(\Om)$, we have
\begin{equation}\label{subsolution}
\int_{\Omega}|\nabla u|^{p-2}\nabla u\nabla\phi\,dx+\int\limits_{\mathbb{R}^N}\int\limits_{\mathbb{R}^N}\mathcal{A}\big((u(x,y)\big){\big(\phi(x)-\phi(y)\big)}\,d\mu\leq\int_{\Omega}\frac{f(x)}{u(x)^\delta}\phi(x)\,dx.
\end{equation}
\end{Definition}

\begin{Remark}\label{defrmk1}
Note that Lemma \ref{defineq} and Lemma \ref{locnon1} ensures that Definition \ref{supsolution}-\ref{subsolution} makes sense.
\end{Remark}

Let us fix a weak supersolution $v$ of the problem $(\mathcal{S})$ and define the following nonempty, closed and convex subset of $W_0^{1,p}(\Om)$ given by
$$
\mathcal{K} := \{\phi\in W_0^{1,p}(\Om):0\leq\phi\leq v\text{ in }\Omega\}.
$$

For $k>0$, we define the truncated function
\begin{equation*}
  g_k(l) := 
  \begin{cases}
    \text{min}\{l^{-\delta},k\}, & \text{ if } l>0,\\
    k, & \text{ if } l\leq 0.
  \end{cases}
\end{equation*}

Let $\mathcal{G}_k$ be the primitive of $g_k$ and we define the functional $J_k:W_0^{1,p}(\Om)\to[-\infty,+\infty]$ by
$$
J_k(\phi):=\frac{1}{p}\int_{\Om}|\nabla\phi|^p\,dx+\frac{1}{p}\int\limits_{\mathbb{R}^N}\int\limits_{\mathbb{R}^N}\frac{|\phi(x)-\phi(y)|^p}{|x-y|^{N+ps}}\,dx dy-\int_{\Om}f(x)\mathcal{G}_k(\phi)\,dx.
$$

Note that $J_k$ is coercive. Indeed, for $1<p<N$ with $f\in L^t(\Omega)$ where $t=(p^*)'$, we have
$$
\left|\int_{\Om}f(x)\mathcal{G}_k(\phi)\,dx\right|\leq k\int_{\Om}f|\phi|\,dx\leq k\|f\|_{L^t(\Om)}\|\phi\|_{L^{p^*}(\Om)}\leq kC_0\|f\|_{L^t(\Om)}\left(\int_{\Om}|\nabla\phi|^p\,dx\right)^\frac{1}{p},
$$
for every $\phi\in\mathcal{K}$, $C_0$ is the Sobolev constant obtained by Lemma \ref{embedding}. Similar estimate holds for $p\geq N$. Hence for somce constant $C>0$, we have
$$
J_k(\phi)\geq\frac{1}{p}||\phi||^p-C\|\phi\|,
$$
where $\|\cdot\|$ is given by the gradient norm \eqref{equinorm}. Therefore, $J_k$ is coercive over $\mathcal{K}$. Arguing as in the proof of Lemma \ref{auxresult}, we get $J_k$ is weakly lower semicontinuous over $\mathcal{K}$. Therefore there exists a minimum of $J_k$, say $z\in\mathcal{K}$ such that for every $\psi\in z+\big(W_0^{1,p}(\Om)\cap L^\infty_{c}(\Om)\big)$ with $\psi\in\mathcal{K}$, we have
\begin{equation}\label{ineq2}
\begin{split}
&\int_{\Omega}|\nabla z|^{p-2}\nabla z\nabla(\psi-z)\,dx+\int\limits_{\mathbb{R}^N}\int\limits_{\mathbb{R}^N}\mathcal{A}\big((z(x,y)\big)((\psi-z)(x)-(\psi-z)(y))\,d\mu\\
&\quad\geq\int_{\Omega}f(x)\mathcal{G}_k'(z)(\psi-z)\,dx,
\end{split}
\end{equation}
where $L^\infty_{c}(\Om)$ denotes the set of bounded functions with compact support in $\Om$. Now we establish the following variational inequality.

\begin{Lemma}\label{unilemma1}
For every nonnegative $\phi\in C_c^{1}(\Omega)$, we have
\begin{equation}\label{ineq1}
\int_{\Omega}|\nabla z|^{p-2}\nabla z\nabla\phi\,dx+\int\limits_{\mathbb{R}^N}\int\limits_{\mathbb{R}^N}\mathcal{A}\big((z(x,y)\big)(\phi(x)-\phi(y))\,d\mu\geq\int_{\Omega}f(x)\mathcal{G}_k'(z)\phi\,dx,
\end{equation}
where $z\in\mathcal{K}$ is given by \eqref{ineq2}.
\end{Lemma}

\begin{proof}
Let us consider a real valued function $\xi\in C_c^{\infty}(\mathbb{R})$ such that $0\leq\xi\leq 1$ in $\mathbb{R}$, $\xi\equiv 1$ in $[-1,1]$ and $\xi\equiv 0$ in $(-\infty,-2]\cup[2,\infty)$. Define the function $\phi_h:=\xi(\frac{z}{h})\phi$ and $\phi_{h,t}:=\text{min}\{z+t\phi_h,v\}$ with $h\geq 1$ and $t>0$ for a given nonnegative $\phi\in C_c^{1}(\Omega)$. 
Then by the inequality (\ref{ineq2}), we have
\begin{equation}\label{ineq3}
\begin{split}
&\int_{\Omega}|\nabla z|^{p-2}\nabla z\nabla(\phi_{h,t}-z)\,dx+\int\limits_{\mathbb{R}^N}\int\limits_{\mathbb{R}^N}\mathcal{A}\big((z(x,y)\big))((\phi_{h,t}-z)(x)-(\phi_{h,t}-z)(y))\,d\mu\\
&\quad\geq\int_{\Omega}f(x)\mathcal{G}_k'(z)(\phi_{h,t}-z)\,dx.
\end{split}
\end{equation}

We define,
\begin{align*}
I:&=C\int_{\Omega}|\nabla(\phi_{h,t}-z)|^{2}(|\nabla\phi_{h,t}|+|\nabla z|)^{p-2}\,dx\\
&+C\int\limits_{\mathbb{R}^N}\int\limits_{\mathbb{R}^N}\frac{\big(|\phi_{h,t}(x)-\phi_{h,t}(y)|+|z(x)-z(y)|\big)^{p-2}\big((\phi_{h,t}-z)(x)-(\phi_{h,t}-z)(y)\big)^2}{|x-y|^{N+ps}}\,dx dy,
\end{align*}
 where $C$ is given by Lemma \ref{AI}. First applying Lemma \ref{AI} and then by using \eqref{ineq3}, we have
\begin{multline*}
I\leq \int_{\Omega}\left(|\nabla\phi_{h,t}|^{p-2}\nabla\phi_{h,t}-|\nabla z|^{p-2}\nabla z\right)\nabla(\phi_{h,t}-z)\,dx\\
\quad+\int\limits_{\mathbb{R}^N}\int\limits_{\mathbb{R}^N}\big(\mathcal{A}(\phi_{h,t}(x,y))-\mathcal{A}(z(x,y))\big)\big((\phi_{h,t}-z)(x)-(\phi_{h,t}-z)(y)\big)\,d\mu\\
\leq \int_{\Omega}|\nabla\phi_{h,t}|^{p-2}\nabla\phi_{h,t}\nabla(\phi_{h,t}-z)\,dx
+\int\limits_{\mathbb{R}^N}\int\limits_{\mathbb{R}^N}\mathcal{A}\big((\phi_{h,t}(x,y)\big)\big((\phi_{h,t}-z)(x)-(\phi_{h,t}-z)(y)\big)\,d\mu\\
\quad\quad-\int_{\Om}f(x)\mathcal{G}_k'(z)(\phi_{h,t}-z)\,dx.
\end{multline*}
Therefore, 
\begin{multline}
\label{ineq4}
I-\int_{\Omega}f(x)\big(\mathcal{G}_k'(\phi_{h,t})-\mathcal{G}_k'(z)\big)(\phi_{h,t}-z)\,dx
\leq\int_{\Omega}|\nabla\phi_{h,t}|^{p-2}\nabla\phi_{h,t}\nabla(\phi_{h,t}-z)\,dx\\
+\int\limits_{\mathbb{R}^N}\int\limits_{\mathbb{R}^N}\mathcal{A}\big(\phi_{h,t}(x,y)\big)\big((\phi_{h,t}-z)(x)-(\phi_{h,t}-z)(y)\big)\,d\mu
-\int_{\Omega}f(x)\mathcal{G}_k'(\phi_{h,t})(\phi_{h,t}-z)\,dx\\
=\Big\{\int_{\Omega}g(x)\,dx+\int\limits_{\mathbb{R}^N}\int\limits_{\mathbb{R}^N}h(x,y)\,dx dy-\int_{\Omega}f(x)\mathcal{G}_k{'}(\phi_{h,t})(\phi_{h,t}-z-t\phi_h)\,dx\Big\}\\
+t\Big\{\int_{\Omega}|\nabla\phi_{h,t}|^{p-2}\nabla\phi_h\,dx
+\int\limits_{\mathbb{R}^N} \int\limits_{\mathbb{R}^N} \mathcal{A}\big(\phi_{h,t}(x,y)\big) \big(\phi_h(x)-\phi_h(y)\big)\,d\mu
-\int_{\Omega}f(x)\mathcal{G}_k'(\phi_{h,t})\phi_h\,dx\Big\},
\end{multline}
where 
$$
g(x)=|\nabla\phi_{h,t}|^{p-2}\nabla\phi_{h,t}\nabla(\phi_{h,t}-z-t\phi_h)
$$
and
$$
h(x,y)=\frac{\mathcal{A}\big(\phi_{h,t}(x,y)\big)\big((\phi_{h,t}-z-t\phi_h)(x)-(\phi_{h,t}-z-t\phi_h)(y)\big)}{|x-y|^{N+ps}}.$$
Let us denote by 
$$
g_v(x)=|\nabla v|^{p-2}\nabla v\nabla(\phi_{h,t}-z-t\phi_h)
$$
and
$$
h_v(x,y)=\frac{\mathcal{A}\big((v(x,y)\big)\big((\phi_{h,t}-z-t\phi_h)(x)-(\phi_{h,t}-z-t\phi_h)(y)\big)}{|x-y|^{N+ps}}.
$$
Now following the exact arguments as in the proof of \cite[Lemma $4.6$]{G} and \cite[Lemma $4.1$]{Caninononloc} we obtain
$$
\int_{\Om}g(x,y)\,dx=\int_{\Om}g_v(x)\,dx\text{ and }\int\limits_{\mathbb{R}^N}\int\limits_{\mathbb{R}^N}h(x,y)\,dx dy=\int\limits_{\mathbb{R}^N}\int\limits_{\mathbb{R}^N}h_v(x,y)\,dx dy.
$$
Noting the above property along with the fact that $v$ is a weak supersolution of $(\mathcal{S})$, we choose $(z+t\phi_h-\phi_{h,t})$ as a test function in (\ref{supsolution}) and using
\begin{align*}
\phi_{h,t}=
\begin{cases}
v \text{ on }S_v:=\{x\in\Omega:z(x)+t\phi_h(x)\geq v(x)\},\\
z+t\phi_h\text{ on }S_v^{c},
\end{cases}
\end{align*}
we obtain
\begin{multline*}
\int_{\Om}g(x)\,dx+\int\limits_{\mathbb{R}^N}\int\limits_{\mathbb{R}^N}h(x,y)\,dx dy-\int_{\Om}f(x)\mathcal{G}_k'(\phi_{h,t})(\phi_{h,t}-z-t\phi_h)\,dx\\
\quad=\int_{\Om}g_v(x)\,dx+\int\limits_{\mathbb{R}^N}\int\limits_{\mathbb{R}^N}h_v(x,y)\,dx dy-\int_{\Om}f(x)\mathcal{G}_k'(\phi_{h,t})(\phi_{h,t}-z-t\phi_h)\,dx\leq 0.
\end{multline*}
Using the above fact in \eqref{ineq4} with the observation that $I\geq 0$ and $\phi_{h,t}-z\leq t\phi_h$, we have
\begin{equation}\label{tlim}
\begin{split}
\int_{\Omega}|\nabla &\phi_{h,t}|^{p-2}\nabla\phi_{h,t}\nabla\phi_h\,dx+\int\limits_{\mathbb{R}^N}\int\limits_{\mathbb{R}^N}\mathcal{A}\big((\phi_{h,t}(x,y)\big)\big((\phi_h(x)-\phi_h(y)\big)\,d\mu\\
&\quad-\int_{\Omega}f(x)\mathcal{G}_k'(\phi_{h,t})\phi_h\,dx\geq -\int_{\Omega}f|\mathcal{G}_k'(\phi_{h,t})-\mathcal{G}_k'(z)|\phi_h\,dx.
\end{split}
\end{equation}
Noting $z\in W_0^{1,p}(\Om)$ and recalling the definition of $\phi_{h,t}$, by the Lebesgue dominated convergence theorem, letting $t\to 0$ in \eqref{tlim}, we obtain
\begin{align*}
\int_{\Omega}|\nabla z|^{p-2}\nabla z\nabla\phi_h\,dx+\int\limits_{\mathbb{R}^N}\int\limits_{\mathbb{R}^N}\mathcal{A}\big((z(x,y)\big)\big((\phi_h(x)-\phi_h(y)\big)\,d\mu\geq\int_{\Omega}f(x)\mathcal{G}_k'(z)\phi_{h}\,dx.
\end{align*}
Now, passing the limit as $h\to\infty$ in the above inequality, the result follows.
\end{proof}

Now we have the following comparison principle.
\begin{Lemma}\label{cp}
Suppose that $u$ is a weak subsolution of $(\mathcal{S})$ such that $u\leq 0$ on $\partial\Om$ and $v$ be a weak supersolution of $(\mathcal{S})$. Then $u\leq v$ in $\Om$.
\end{Lemma}

\begin{proof}
Let $k>0$ and $\epsilon=2k^{-\frac{1}{\delta}}$. Since $u\leq 0$ on $\partial\Omega$, for $z\in\mathcal{K}$ as given by Lemma \ref{unilemma1}, we have $(u-z-\epsilon)^{+}\in W_0^{1,p}(\Om)$. For $\eta>0$, let us define
$T_{\eta}(s):=\min\{s,\eta\}$, if $s\geq 0$ and $T_\eta(-s)=-T_{\eta}(s)$, if $s<0$. Applying Lemma \ref{unilemma1}, for any $\eta>0$, by density arguments, it follows that
\begin{equation}
\label{Ineq1}
\begin{split}
&\int_{\Omega}|\nabla z|^{p-2}\nabla z\nabla T_{\eta}\big((u-z-\epsilon)^{+}\big)\,dx\\
&\quad+\int\limits_{\mathbb{R}^N}\int\limits_{\mathbb{R}^N}\mathcal{A}\big(z(x,y)\big)\big(T_{\eta}((u-z-\epsilon)^+(x))-T_{\eta}((u-z-\epsilon)^+(y))\big)\,d\mu\\
&\quad\quad\geq \int_{\Omega}f(x)\mathcal{G}_k'(z)T_{\eta}\big((u-z-\epsilon)^{+}\big)\,dx.
\end{split}
\end{equation}
Since $(u-z-\epsilon)^{+}\in W_0^{1,p}(\Om)$, there exists a sequence $\phi_n\in C_{c}^\infty(\Omega)$ such that $\phi_n\to (u-z-\epsilon)^{+}$ strongly in $W_0^{1,p}(\Om)$. We define $$
\psi_{n,\eta}:=T_{\eta}\big(\text{min}\{(u-z-\epsilon)^{+},\phi_{n}^{+}\}\big)\in W_0^{1,p}(\Om)\cap L_c^{\infty}(\Omega).
$$
Since $u$ is a weak subsolution of $(\mathcal{S})$, we have
\begin{align*}
\int_{\Omega}|\nabla u|^{p-2}\nabla u\nabla\psi_{n,\eta}\,dx+\int\limits_{\mathbb{R}^N}\int\limits_{\mathbb{R}^N}\mathcal{A}\big(u(x,y)\big)\big(\psi_{n,\eta}(x)-\psi_{n,\eta}(y))\big)\,d\mu\leq\int_{\Omega}\frac{f}{u^\delta}\psi_{n,\eta}\,dx.
\end{align*}
As $|\nabla u|^p$ is integrable in the support of $(u-z-\epsilon)^+$, passing the limit as $n\to\infty$,
\begin{equation}\label{Ineq2}
\begin{split}
&\int_{\Omega}|\nabla u|^{p-2}\nabla u\nabla T_{\eta}\big((u-z-\epsilon)^{+}\big)\,dx\\
&\quad+\int\limits_{\mathbb{R}^N}\int\limits_{\mathbb{R}^N}\mathcal{A}\big(u(x,y)\big)\big(T_{\eta}((u-z-\epsilon)^+(x))-
T_{\eta}((u-z-\epsilon)^+(x))\big)\,d\mu\\
&\quad\quad\leq\int_{\Omega}\frac{f}{u^\delta}T_{\eta}(u-z-\epsilon)^{+}\,dx.
\end{split}
\end{equation}
Following the proof of \cite[Theorem 4.2]{Caninononloc} we have
\begin{equation}\label{ueqn}
\begin{split}
&\mathcal{A}(u(x,y))\big(T_{\eta}((u-z-\epsilon)^+(x))-
T_{\eta}((u-z-\epsilon)^+(x))\big)\\
&=\mathcal{A}(u(x,y))\big((u-z)(x)-(u-z)(y)\big)H(x,y),
\end{split}
\end{equation}
and
\begin{equation}\label{zeqn}
\begin{split}
&\mathcal{A}(z(x,y))\big(T_{\eta}((u-z-\epsilon)^+(x))-
T_{\eta}((u-z-\epsilon)^+(x))\big)\\
&=\mathcal{A}(z(x,y))\big((u-z)(x)-(u-z)(y)\big)H(x,y),
\end{split}
\end{equation}
with
$$
H(x,y):=\frac{T_{\eta}((u-z-\epsilon)^{+}(x))-T_{\eta}((u-z-\epsilon)^{+}(y))}{(u-z)(x)-(u-z)(y)},
$$
where $(u-z)(x)-(u-z)(y)\neq 0$. Subtracting (\ref{Ineq1}) and (\ref{Ineq2}) and then, using \eqref{ueqn}, \eqref{zeqn} along with Lemma \ref{AI}, we have
\begin{multline*}
C\int_{\Omega}|\nabla T_{\eta}\big((u-z-\epsilon)^{+}\big)|^{2}(|\nabla u|+|\nabla z|)^{p-2}\,dx\\
+C\int\limits_{\mathbb{R}^N}\int\limits_{\mathbb{R}^N}\frac{\big(|u(x)-u(y)|+|z(x)-z(y)|\big)^{p-2}\big((u-z)(x)-(u-z)(y)\big)^2}{|x-y|^{N+ps}}H(x,y)\,dx dy\\
\leq\int_{\Omega}(|\nabla u|^{p-2}\nabla u-|\nabla z|^{p-2}\nabla z)\nabla T_{\eta}\big((u-z-\epsilon)^{+}\big)\,dx\\
+\int\limits_{\mathbb{R}^N}\int\limits_{\mathbb{R}^N}\big(\mathcal{A}(u(x,y))-\mathcal{A}(z(x,y))\big)\big(T_{\eta}((u-z-\epsilon)^{+}(x))-T_{\eta}((u-z-\epsilon)^{+}(y))\big)\,d\mu\\
\leq \int_{\Omega}f(x)\left(\frac{1}{u^\delta}-\mathcal{G}_k'(z)\right)T_{\eta}\big((u-z-\epsilon)^{+}\big)\,dx\\
\leq\int_{\Omega}f(x)\big(\mathcal{G}_k'(u)-\mathcal{G}_k'(z)\big)T_{\eta}\big((u-z-\epsilon)^{+}\big)\,dx\leq 0.
\end{multline*}
 We note that in the final estimate above, we have used $\epsilon>k^{-\frac{1}{\delta}}$, the definition of $g_k$ and the fact that $u\geq\epsilon$ in the support of $(u-z-\epsilon)^+$. Therefore, using the nonnegativity of $H$, letting $\eta\to\infty$, the above estimate yields 
$$
\int_{\Omega}|\nabla \big((u-z-\epsilon)^{+}\big)|^{2}(|\nabla u|+|\nabla z|)^{p-2}\,dx=0.
$$
As a consequence, $u\leq z+2k^{-\frac{1}{\delta}}\leq v+2k^{-\frac{1}{\delta}}.$ Letting $k\to\infty$, we get $u\leq v$ in $\Omega$.
\end{proof}

\section{Preliminaries for the mixed Sobolev inequalities}

Throughout this section, we assume $0<\delta<1$, $f\in L^m(\Om)\setminus\{0\}$, where $m$ is given by \eqref{m} unless otherwise mentioned. For $v\in W_0^{1,p}(\Om)$ by $\|v\|$, we mean the following norm as defined by \eqref{norm}
\begin{equation}\label{Sobnorm}
\|v\|:=\|v\|_{W_0^{1,p}(\Om)}:=\left(\int_{\Omega}|\nabla v|^p\,dx+\int\limits_{\mathbb{R}^N}\int\limits_{\mathbb{R}^N}\frac{|v(x)-v(y)|^p}{|x-y|^{N+ps}}\,dx\,dy\right)^\frac{1}{p}.
\end{equation}
Let $u_n\in W_0^{1,p}(\Om)$ be the solution of the problem $(\mathcal{A})$ as given by Lemma \ref{approx} and denote by $u_\delta$ to be the pointwise limit of $u_n$ in $\Om$. Note that, by Theorem \ref{thm1}, $u_\delta\in W_0^{1,p}(\Om)$ is the weak solution of $(\mathcal{S})$. Moreover, by Remark \ref{rmkapprox}, we have $u_n\leq u_\delta$ in $\Om$, for all $n\in\mathbb{N}$. Next, we obtain some auxiliary results, which are very useful to prove Theorem \ref{thm5}.

First, we prove the following result, which allows us to choose test functions from the space $W_0^{1,p}(\Om)$ in the equation $\mathcal{(W)}$.

\begin{Lemma}\label{testfn}
Let $\delta>0$, $f\in L^1(\Omega)\setminus\{0\}$ be nonnegative and $u\in W_0^{1,p}(\Omega)$ be a weak solution of the problem $(\mathcal{S})$, then $(\mathcal{W})$ holds for every $\phi\in W_0^{1,p}(\Omega)$.
\end{Lemma}
\begin{proof}
Let $u\in W_0^{1,p}(\Omega)$ solves the problem $(\mathcal{S})$. Therefore, for every $\phi\in C_c^1(\Omega)$, we have
\begin{equation}\label{smthtest}
\int_{\Omega}|\nabla u|^{p-2}\nabla u\cdot\nabla\phi\,dx+\int\limits_{\mathbb{R}^N}\int\limits_{\mathbb{R}^N}\mathcal{A}\big(u(x,y)\big){\big(\phi(x)-\phi(y)\big)}\,d\mu=\int_{\Omega}\frac{f(x)}{u(x)^\delta}\phi(x)\,dx.
\end{equation}
By density, for every $\psi\in W_0^{1,p}(\Omega)$, there exists a sequence of functions $0\leq \psi_n\in C_c^{1}(\Omega)$ such that $\psi_n\to |\psi|$ strongly in $W_0^{1,p}(\Omega)$ as $n\to\infty$ and pointwise almost everywhere in $\Omega$. We observe that
\begin{multline}
\label{unique}
\Big|\int_{\Omega}\frac{f(x)}{u(x)^\delta}\psi\,dx\Big|\leq \int_{\Omega}\frac{f(x)}{u(x)^\delta}|\psi|\,dx\leq\liminf_{n\to\infty}\int_{\Omega}\frac{f(x)}{u(x)^\delta}\psi_n\,dx
=\liminf_{n\to\infty}\langle -\Delta_p u+(-\Delta_p)^s u,\psi_n \rangle\\
=\int_{\Omega}|\nabla u|^{p-2}\nabla u\cdot\nabla\psi_n\,dx+\int\limits_{\mathbb{R}^N}\int\limits_{\mathbb{R}^N}\mathcal{A}\big(u(x,y)\big){\big(\psi_n(x)-\psi_n(y)\big)}\,d\mu\\
\leq C\|u\|^{p-1}\lim_{n\to\infty}\|\psi_n\|\leq C\|u\|^{p-1}\||\psi|\|
\leq C\|u\|^{p-1}\|\psi\|,
\end{multline}
for some positive constant $C$ (independent of $n$). Let $\phi\in W_0^{1,p}(\Omega)$, then there exists $\phi_n\in C_c^{1}(\Omega)$ converges to $\phi$ strongly in $W_0^{1,p}(\Omega)$. Now, using $\psi=\phi_n-\phi$ in \eqref{unique}, we obtain
\begin{equation}\label{lhslim}
\lim_{n\to\infty}\Big|\int_{\Omega}\frac{f(x)}{u(x)^\delta}(\phi_n-\phi)\,dx\Big|\leq C\|u\|^{p-1}\lim_{n\to\infty}\|\phi_n-\phi\|=0.
\end{equation}
Again, since $\phi_n\to\phi$ strongly in $W_0^{1,p}(\Omega)$ as $n\to\infty$, we have
\begin{equation}\label{rhslim}
\lim_{n\to\infty}\left\{\int_{\Omega}|\nabla u|^{p-2}\nabla u\cdot\nabla(\phi_n-\phi)\,dx+\int\limits_{\mathbb{R}^N}\int\limits_{\mathbb{R}^N}\mathcal{A}\big(u(x,y)\big){\big((\phi_n-\phi)(x)-(\phi_n-\phi)(y)\big)}\,d\mu\right\}=0.
\end{equation}
Hence, using \eqref{lhslim} and \eqref{rhslim} in \eqref{smthtest} the result follows.
\end{proof}

As a consequence of Lemma \ref{testfn}, we have the following simple proof of the uniqueness result.

\begin{Corollary}\label{tstcor}
Let $\delta>0$ and $f\in L^1(\Om)\setminus\{0\}$ be nonnegavite. Then, the problem $\mathcal{(S)}$ admits at most one weak solution in $W_0^{1,p}(\Om)$.
\end{Corollary}
\begin{proof}
By contradiction, suppose $u_1,u_2\in W_0^{1,p}(\Om)$ are two weak solutions of the problem $\mathcal{(S)}$. Then, by Lemma \ref{testfn}, for every $\phi\in W_0^{1,p}(\Om)$, we have
\begin{equation}\label{tstfnap1}
\int_{\Omega}|\nabla u_1|^{p-2}\nabla u_1\cdot\nabla\phi\,dx+\int\limits_{\mathbb{R}^N}\int\limits_{\mathbb{R}^N}\mathcal{A}\big(u_1(x,y)\big){\big(\phi(x)-\phi(y)\big)}\,d\mu=\int_{\Omega}\frac{f(x)}{u_1(x)^\delta}\phi(x)\,dx,
\end{equation}

\begin{equation}\label{tstfnap2}
\int_{\Omega}|\nabla u_2|^{p-2}\nabla u_2\cdot\nabla\phi\,dx+\int\limits_{\mathbb{R}^N}\int\limits_{\mathbb{R}^N}\mathcal{A}\big(u_2(x,y)\big){\big(\phi(x)-\phi(y)\big)}\,d\mu=\int_{\Omega}\frac{f(x)}{u_2(x)^\delta}\phi(x)\,dx.
\end{equation}
Therefore, first choosing $\phi=(u_1-u_2)^+\in W_0^{1,p}(\Om)$ in \eqref{tstfnap1} and \eqref{tstfnap2} and then subtracting them, we have
\begin{equation}\label{tstfnap}
\begin{split}
&\int_{\Omega}\left\{|\nabla u_1|^{p-2}\nabla u_1-|\nabla u_1|^{p-2}\nabla u_1\right\}\nabla(u_1-u_2)^+\,dx\\
&\quad+\int\limits_{\mathbb{R}^N}\int\limits_{\mathbb{R}^N}\left\{\mathcal{A}\big(u_1(x,y)\big)-\mathcal{A}\big(u_2(x,y)\big)\right\}{\big((u_1-u_2)^+(x)-(u_1-u_2)^+(y)\big)}\,d\mu\\
&=\int_{\Omega}f(x)\left\{\frac{1}{u_1(x)^\delta}-\frac{1}{u_2(x)^\delta}\right\}(u_1-u_2)^+(x)\,dx\leq 0.
\end{split}
\end{equation}
Again, following the same arguments from the proof of \cite[Lemma $9$]{Ling}, we obtain
\begin{equation}\label{tstfnap3}
\int\limits_{\mathbb{R}^N}\int\limits_{\mathbb{R}^N}\left\{\mathcal{A}\big(u_1(x,y)\big)-\mathcal{A}\big(u_2(x,y)\big)\right\}{\big((u_1-u_2)^+(x)-(u_1-u_2)^+(y)\big)}\,d\mu\geq 0.
\end{equation}
Hence using \eqref{tstfnap3} in \eqref{tstfnap} and then applying Lemma \ref{AI}, we have $u_2\geq u_1$ in $\Om$. Similarly, we obtain $u_1\geq u_2$ in $\Om$. Hence, the result follows.
\end{proof} 

\begin{Remark}\label{corrmk}
 Corollary \ref{tstcor} extends Theorem \ref{thm4} to the case of $f\in L^1(\Omega)$ to obtain uniqueness in $W_0^{1,p}(\Om)$.
\end{Remark}

\begin{Lemma}\label{lemma1}
Let $n\in\mathbb{N}$. Then for every $\phi\in W_0^{1,p}(\Omega)$, we have
\begin{equation}\label{prop1}
\|u_n\|^{p}\leq\|\phi\|^{p}+p\int_{\Omega}\frac{(u_{n}-\phi)}{(u_{n}+\frac{1}{n})^\delta}f_{n}\,dx.
\end{equation}
Moreover, we have
\begin{equation}\label{nmon}
\|u_n\|\leq \|u_{n+1}\|\text{ for every }n\in\mathbb{N}.
\end{equation}
\end{Lemma}

\begin{proof}
Let $h\in W_0^{1,p}(\Om)$, then by Lemma \ref{auxresult}, there exists a unique solution $v\in W_0^{1,p}(\Om)$ to the problem
\begin{equation*}
-\Delta_p v+(-\Delta_p)^s v=\frac{f_{n}(x)}{(h^{+}+\frac{1}{n})^\delta},\,v>0\text{ in }\Omega,\,v=0\text{ in }\mathbb{R}^N\setminus\Omega.
\end{equation*}

Moreover, we observe that $v$ is a minimizer of the functional $J:W_0^{1,p}(\Om)\to\mathbb{R}$ given by
$$
J(\phi):=\frac{1}{p}\|\phi\|^{p}-\int_{\Omega}\frac{f_{n}}{(h^{+}+\frac{1}{n})^\delta}\phi\,dx.
$$

Therefore, for every $\phi\in W_0^{1,p}(\Om)$, we have $J(v)\leq J(\phi)$. Hence, we have
\begin{equation}\label{mineqn}
\frac{1}{p}\|v\|^{p}-\int_{\Omega}\frac{f_{n}}{(h^{+}+\frac{1}{n})^\delta}v\,dx
\leq \frac{1}{p}\|\phi\|^{p}-\int_{\Omega}\frac{f_{n}}{(h^{+}+\frac{1}{n})^\delta}\phi\,dx.
\end{equation}

Then the inequality \eqref{prop1} follows by choosing $v=h=u_{n}$ in the inequality \eqref{mineqn}. Now, choosing $\phi=u_{n+1}$ in \eqref{prop1} and using the monotone property $u_{n}\leq u_{n+1}$ from Lemma \ref{approx}, we obtain $\|u_{n}\|\leq \|u_{n+1}\|$.
\end{proof} 

\begin{Lemma}\label{strong}
Upto a subsequence $\{u_n\}$ converges to $u_\delta$ strongly in $W_0^{1,p}(\Omega).$
\end{Lemma}
\begin{proof}
Note that $u_n\leq u_\delta$. Thus, choosing $\phi=u_\delta$ in \eqref{prop1}, we obtain
$$
\|u_n\|\leq\|u_\delta\|,
$$
which gives from the monotone property \eqref{nmon} of Lemma \ref{lemma1}, 
\begin{equation}\label{lim1}
\lim_{n\to\infty}\|u_n\|\leq\|u_\delta\|.
\end{equation}
By Lemma \ref{unibddless}, the sequence $\{u_n\}$ is uniformly bounded in $W_0^{1,p}(\Om)$ and thus, upto a subsequence, $u_n\rightharpoonup u_\delta$ weakly in $W_0^{1,p}(\Omega)$. Therefore
\begin{equation}\label{lim2}
\|u_\delta\|\leq \lim_{n\to\infty}\|u_n\|.
\end{equation}

Hence from \eqref{lim1}, \eqref{lim2} and the uniform convexity of $W_0^{1,p}(\Om)$, the result follows.
\end{proof}

\begin{Corollary}\label{strmk}
As a consequence of Lemma \ref{strong}, upto a subsequence $\nabla u_n\to\nabla u_\delta$ pointwise almost everywhere in $\Om$, which coincides with Theorem \ref{grad}, provided $0<\delta<1$.
\end{Corollary}

\begin{Lemma}\label{minprop}
Let $I_{\delta}:W_0^{1,p}(\Omega)\to\mathbb{R}$ be a functional defined by
$$
I_{\delta}(v):=\frac{1}{p}\|v\|^p-\frac{1}{1-\delta}\int_{\Omega}(v^{+})^{1-\delta}f\,dx.
$$
Then $u_{\delta}$ is a minimizer of the functional $I_{\delta}.$
\end{Lemma}
\begin{proof}
We define the auxiliary functional $I_{n}:W_0^{1,p}(\Omega)\to\mathbb{R}$ by
$$
I_{n}(v):=\frac{1}{p}\|v\|^p-\int_{\Omega}G_n(v)f_{n}\,dx,
$$
where
$$
G_n(t):=\frac{1}{1-\delta}\Big(t^{+}+\frac{1}{n}\Big)^{1-\delta}-\Big(\frac{1}{n}\Big)^{-\delta}t^{-}.
$$
We observe that $I_{n}$ is bounded below, coercive and a $C^1$ functional. Hence, $I_{n}$ consists a minimizer $v_{n}\in W_0^{1,p}(\Omega)$ such that
$$
\langle I_{n}{'}(v_{n}),\phi\rangle=0,\text{ for all }\phi\in W_0^{1,p}(\Omega).
$$
Now the fact $I_n(v_n)\leq I_n(v_n^{+})$ gives $v_n\geq 0$ in $\Om$. Therefore, $v_{n}$ solves the approximated problem $\mathcal{(A)}$. By the uniqueness result in Lemma \ref{approx}, we get $u_{n}=v_{n}$ and so $u_{n}$ is a minimizer of $I_{n}$. Therefore, for every $v\in W_0^{1,p}(\Om)$, we have
\begin{equation}\label{min}
I_n(u_n)\leq I_n(v^+).
\end{equation}
Since $u_n\leq u_\delta$, by the Lebesgue dominated convergence theorem, we have
$$
\lim_{n\to\infty}\int_{\Omega}G_n(u_{n})f_{n}\,dx=\frac{1}{1-\delta}\int_{\Omega}u_{\delta}^{1-\delta}f\,dx.
$$
Moreover, by Lemma \ref{strong}, we have
$$
\lim_{n\to\infty}\|u_{n}\|=\|u_\delta\|.
$$
Hence, we have
\begin{equation}\label{newlim2}
\lim_{n\to\infty}I_{n}(u_{n})=I_\delta(u_\delta).
\end{equation}
Furthermore, for every $v\in W_0^{1,p}(\Omega)$, we get
\begin{equation}\label{lim3}
\lim_{n\to\infty}\int_{\Omega}G_n(v^{+})f_{n}\,dx=\frac{1}{1-\delta}\int_{\Omega}(v^{+})^{1-\delta}f\,dx.
\end{equation}
Noting $\|v^+\|\leq\|v\|$ and using \eqref{newlim2} and \eqref{lim3} in \eqref{min}, we arrive at $I_\delta(u_\delta)\leq I_\delta(v)$, for all $v\in W_0^{1,p}(\Omega)$. Hence, the result follows.
\end{proof}

\section{Proof of the main results}
\subsection{Proof of the existence results}
\textbf{Proof of Theorem \ref{thm1}:} Let $f\in L^m(\Om)\setminus\{0\}$, where $m$ is given by \eqref{m}. Then, by Lemma \ref{unibddless}, the sequence $\{u_n\}$ is uniformly bounded in $W_0^{1,p}(\Om)$. Hence the pointwise limit $u\in W_0^{1,p}(\Om)$ and thus belong to $L^{p-1}(\Om)$. Moreover, using Theorem \ref{grad}, upto a subsequence
$$
\nabla u_n\to \nabla u \text{ pointwise almost everywhere in }\Om.
$$
As a consequence, for every $\phi\in C_c^1(\Om)$, we have
\begin{equation}\label{loclim}
\lim_{n\to\infty}\int_{\Om}|\nabla u_n|^{p-2}\nabla u_n\nabla\phi\,dx=\int_{\Om}|\nabla u|^{p-2}\nabla u\nabla\phi\,dx.
\end{equation}
Since $\phi\in C_c^1(\Om)$ and $\{u_n\}$ is uniformly bounded in $W_0^{1,p}(\Om)$, by Lemma \ref{locnon1}
$$
\frac{|u_n(x)-u_n(y)|^{p-2}\big(u_n(x)-u_n(y)\big)}{|x-y|^\frac{N+ps}{p'}}\in L^{p'}(\mathbb{R}^N\times\mathbb{R}^N),
$$
is uniformly bounded and
$$
\frac{\phi(x)-\phi(y)}{|x-y|^\frac{N+ps}{p}}\in L^p(\mathbb{R}^N\times\mathbb{R}^N).
$$
Therefore, by the weak convergence, we have
\begin{equation}\label{nonloclim}
\lim_{n\to\infty}\int\limits_{\mathbb{R}^N}\int\limits_{\mathbb{R}^N}\mathcal{A}\big(u_n(x,y)\big)\big(\phi(x)-\phi(y)\big)\,d\mu=\int\limits_{\mathbb{R}^N}\int\limits_{\mathbb{R}^N}\mathcal{A}\big(u(x,y)\big)\big(\phi(x)-\phi(y)\big)\,d\mu.
\end{equation}
By Lemma \ref{approx}, $u_n\geq C(\omega)>0$ on a set $\mathrm{supp}\,\phi=\omega$ for some constant $C(\omega)>0$, independent of $n$. Therefore, for every $\phi\in C_c^{1}(\Om)$, we have
$$
\left|\frac{f_n }{\big(u_n+\frac{1}{n}\big)^\delta}\phi\right|\leq\frac{\|\phi\|_{L^\infty(\Om)}}{C(\omega)^\delta}|f|\text{ in }\Om.
$$
By the Lebesgue dominated convergence theorem, we have
\begin{equation}\label{slim}
\lim_{n\to\infty}\int_{\Om}\frac{f_n}{\big(u_n+\frac{1}{n}\big)^\delta}\phi\,dx=\int_{\Om}\frac{f}{u^\delta}\phi\,dx.
\end{equation}
Using \eqref{loclim}, \eqref{nonloclim} and \eqref{slim} in $(\mathcal{A})$ and noting Remark \ref{bcrmk}, we obtain $u$ is a weak solution of the problem $(\mathcal{S})$. \qed

\textbf{Proof of Theorem \ref{thm2}:} By Lemma \ref{unibdequal}, the sequence $\{u_n\}$ is uniformly bounded in $W_0^{1,p}(\Om)$. Now, proceeding as in the proof of Theorem \ref{thm1}, the result follows. \qed

\textbf{Proof of Theorem \ref{thm3}:} By Lemma \ref{unibddgrt}, the sequence $\left\{u_n^\frac{\delta+p-1}{p}\right\}$ is uniformly bounded in $W_0^{1,p}(\Om)$. Thus $u^\frac{\delta+p-1}{p}\in W_0^{1,p}(\Om)$. This also gives $u\in L^{p-1}(\Om)$. Moreover, by Lemma \ref{unibddgrt}, the sequence $\{u_n\}$ is uniformly bounded in $W_{\mathrm{loc}}^{1,p}(\Om)$. Hence $u\in W^{1,p}_{\mathrm{loc}}(\Om)$. Now following the lines of the proof of Theorem \ref{thm1}, for every $\phi\in C_c^{1}(\Om)$, we obtain
\begin{equation}\label{llim}
\lim_{n\to\infty}\int_{\Om}|\nabla u_n|^{p-2}\nabla u_n\nabla\phi\,dx=\int_{\Om}|\nabla u|^{p-2}\nabla u\nabla\phi\,dx\text{ and }
\end{equation}

\begin{equation}\label{rlim}
\lim_{n\to\infty}\int_{\Om}\frac{f_n}{\big(u_n+\frac{1}{n}\big)^\delta}\phi\,dx=\int_{\Om}\frac{f}{u^\delta}\phi\,dx.
\end{equation}
For the nonlocal part, following the same arguments as in the proof of \cite[Theorem $3.6$]{Caninononloc}, for every $\phi\in C_c^{1}(\Om)$, we have
\begin{equation}\label{nlim}
\lim_{n\to\infty}\int\limits_{\mathbb{R}^N}\int\limits_{\mathbb{R}^N}\mathcal{A}\big(u_n(x,y)\big)\big(\phi(x)-\phi(y)\big)\,d\mu=\int\limits_{\mathbb{R}^N}\int\limits_{\mathbb{R}^N}\mathcal{A}\big(u(x,y)\big)\big(\phi(x)-\phi(y)\big)\,d\mu.
\end{equation}
Hence, from \eqref{llim}, \eqref{rlim} and \eqref{nlim} along with Remark \ref{bcrmk} the result follows. \qed

\textbf{Proof of Theorem \ref{regthm}:} By Lemma \ref{unibddreg}, the sequence $\{u_n\}$ is uniformly bounded in $L^\infty(\Om)$. Hence, $u\in L^\infty(\Om)$. \qed

\subsection{Proof of the uniqueness result}
\textbf{Proof of Theorem \ref{thm4}:} Let $u$ and $v$ are weak solutions of $(\mathcal{S})$ under the zero Dirichlet boundary condition according to the Definition \ref{bc}. Then, it follows that $u$ is a weak subolution such that $u\leq 0$ on $\partial\Om$ and $v$ is a weak supersolution of $(\mathcal{S})$. Therefore, by Lemma \ref{cp}, we obtain $u\leq v$ in $\Om$. Similarly, we get $v\leq u$ in $\Om$. Hence $u\equiv v$ in $\Om$. \qed

\subsection{Proof of the symmetry result}
\textbf{Proof of Theorem \ref{sym}:} Without loss of generality, we assume that $\Om$ is symmetric with respect to the $x_1$ direction and $f(x_1,x')=f(-x_1,x')$ where $x'\in\mathbb{R}^{N-1}$. Then setting $v(x_1,x')=u(-x_1,x')$, we observe that $v$ is again a weak solution of the problem $(\mathcal{S})$. Therefore, by Theorem \ref{thm4}, we obtain $u(x)=v(x)$. Hence the result follows. \qed

\subsection{Proof of the mixed Sobolev inequalities}
\textbf{Proof of Theorem \ref{thm5}:}
Here we consider the norm of a function $v$ in $W_0^{1,p}(\Om)$ defined by \eqref{norm} as follows: 
$$
\|v\|=\left(\int_{\Omega}|\nabla v|^p\,dx+\int\limits_{\mathbb{R}^N}\int\limits_{\mathbb{R}^N}\frac{|v(x)-v(y)|^p}{|x-y|^{N+ps}}\,dx\,dy\right)^\frac{1}{p}.
$$
\begin{enumerate}
\item[(a)] To obtain the result, it is enough to prove that
$$
\mu(\Omega)=\inf_{v\in S_{\delta}}\|v\|^{p}=\|u_\delta\|^{\frac{p(1-\delta-p)}{1-\delta}},
$$
where
$$
S_{\delta}:=\left\{v\in W_0^{1,p}(\Omega):\int_{\Omega}|v|^{1-\delta}f\,dx=1\right\}.
$$
To this end, we consider $V_\delta=\tau_\delta u_\delta\in S_{\delta}$, where
$$
\tau_\delta=\left(\int_{\Omega}u_\delta^{1-\delta}f\,dx\right)^{-\frac{1}{1-\delta}}.
$$
By Lemma \ref{testfn}, choosing $u_\delta$ as a test function in $(\mathcal{W})$, we get
\begin{equation}\label{useful1}
\|u_\delta\|^p=\int_{\Omega}u_\delta^{1-\delta}f\,dx.
\end{equation}
Now using \eqref{useful1} we have
\begin{multline}
\label{extremal}
\|V_\delta\|^{p}=\int_{\Om}|\nabla V_\delta|^p\,dx+\int\limits_{\mathbb{R}^N}\int\limits_{\mathbb{R}^N}\frac{|V_{\delta}(x)-V_\delta(y)|^p}{|x-y|^{N+ps}}\,dx dy\\
=(\tau_\delta)^{p}\|u_\delta\|^{p}
=\Bigg(\int_{\Omega}u_\delta^{1-\delta}f\,dx\Bigg)^{-\frac{p}{1-\delta}}\|u_\delta\|^{p}=\|u_\delta\|^\frac{p(1-\delta-p)}{1-\delta}.
\end{multline}
Let $v\in S_{\delta}$ and $\lambda=\|v\|^{-\frac{p}{p+\delta-1}}$. By Lemma \ref{minprop}, $u_\delta$ minimizes the functional $I_\delta$, which gives $I_\delta(u_\delta)\leq I_\delta(\lambda|v|)$. Therefore, using \eqref{useful1}, we have 
\begin{multline*}
\Big(\frac{1}{p}-\frac{1}{1-\delta}\Big)\|u_\delta\|^{p}=I_\delta(u_\delta)\leq I_\delta(\lambda|v|)=\frac{\lambda^{p}}{p}\||v|\|^{p}-\frac{\lambda^{1-\delta}}{1-\delta}\int_{\Om}|v|^{1-\delta}f\,dx\\
=\frac{\lambda^{p}}{p}\||v|\|^{p}-\frac{\lambda^{1-\delta}}{1-\delta}
\leq\frac{\lambda^{p}}{p}\|v\|^{p}-\frac{\lambda^{1-\delta}}{1-\delta}
=\Big(\frac{1}{p}-\frac{1}{1-\delta}\Big)\|v\|^\frac{p(\delta-1)}{\delta+p-1}.
\end{multline*}
Since $v\in S_{\delta}$ is arbitrary, the above estimate gives
\begin{equation}\label{infprop}
\|u_{\delta}\|^\frac{p(1-\delta-p)}{1-\delta}\leq\inf_{v\in S_{\delta}}\|v\|^{p}.
\end{equation}
As $V_{\delta}\in S_{\delta}$, from \eqref{extremal} and \eqref{infprop}, the result follows. \qed
\item[(b)] Suppose the mixed Sobolev inequality \eqref{inequality2} holds. If $C>\mu(\Omega)$, by (a) and \eqref{extremal}, we have
$$
C\Big(\int_{\Omega}V_{\delta}^{1-\delta}f\,dx\Big)^\frac{p}{1-\delta}>\|V_{\delta}\|^p,
$$
which contradicts our assumption \eqref{inequality2}. Conversely, assume that
$$
C\leq\mu(\Omega)=\inf_{v\in S_{\delta}}\|v\|^{p}\leq\|w\|^{p},
$$ for all $w\in S_{\delta}$. Observe that the claim directly follows if $v\equiv 0$. Thus we only deal with the case of $v\in W_0^{1,p}(\Omega)\setminus\{0\}$ for which we have
$$
w=\Bigg(\int_{\Omega}|v|^{1-\delta}f\,dx\Bigg)^{-\frac{1}{1-\delta}}v\in S_{\delta}.
$$
Therefore, we have
$$
C\leq\Bigg(\int_{\Omega}|v|^{1-\delta}f\,dx\Bigg)^{-\frac{p}{1-\delta}}\|v\|^p, 
$$
which proves the result.
\item[(c)] From \eqref{extremal} in $(a)$, we know that $\mu(\Om)=\|V_\delta\|^p$. Let $v\in S_\delta$ be such that $\mu(\Om)=\|v\|^p$. First, we claim that $v$ has a constant sign in $\Om$. Indeed if $v$ changes sign in $\Om$, then using the fact that 
$$
\big||v(x)|-|v(y)|\big|<|v(x)-v(y)|,
$$
we have
\begin{equation}\label{c1}
\||v|\|^p<\|v\|^p.
\end{equation}
Moreover, since $|v|\in S_\delta$, we have 
$$
\|v\|^p=\mu(\Om)\leq \||v|\|^p,
$$
which contradicts \eqref{c1}. Hence, the claim follows. Without loss of generality, we assume that $v\geq 0$ in $\Om$. Using the fact that $V_\delta, v\geq 0$ and $0<1-\delta<1$, we have
\begin{multline*}
g=\left(\int_{\Om}\Big(\frac{v}{2}+\frac{V_\delta}{2}\Big)^{1-\delta}f\,dx\right)^\frac{1}{1-\delta}
=\left(\int_{\Om}\Big(\frac{v}{2}f^\frac{1}{1-\delta}+\frac{V_\delta}{2}f^\frac{1}{1-\delta}\Big)^{1-\delta}\,dx\right)^\frac{1}{1-\delta}\\
\geq\left(\int_{\Om}\left(\frac{v}{2}\right)^{1-\delta}f\,dx\right)^\frac{1}{1-\delta}+\left(\int_{\Om}\left(\frac{V_\delta}{2}\right)^{1-\delta}f\,dx\right)^\frac{1}{1-\delta}\\
\geq\frac{1}{2}\left(\int_{\Om}v^{1-\delta}f\,dx\right)^\frac{1}{1-\delta}+\frac{1}{2}\left(\int_{\Om}V_{\delta}^{1-\delta}f\,dx\right)^\frac{1}{1-\delta}
=\frac{1}{2}+\frac{1}{2}=1.
\end{multline*}
Then, since $h=\frac{v+V_\delta}{2g}\in S_\delta$, we obtain
$$
\mu(\Om)\leq\|h\|^p\leq\frac{1}{g^p}\left\|\frac{v+V_\delta}{2}\right\|^p\leq\frac{\mu(\Om)}{g^p}\leq\mu(\Om).
$$
Hence, we get $g=1$ and therefore, $\frac{v+V_\delta}{2}\in S_\delta$ and 
$$
\mu(\Om)^\frac{1}{p}=\left\|\frac{v+V_\delta}{2}\right\|=\frac{\|v\|}{2}+\frac{\|V_\delta\|}{2}.
$$
Since the norm 
$
v\to \|v\|
$
is strictly convex, we have $v=V_\delta$. So, $\mu(\Om)=\|v\|^p$ for some $v\in S_\delta$, if and only if $v=V_\delta$ or $-V_\delta$. Thus, if \eqref{sim} holds for some $w\in W_0^{1,p}(\Om)\setminus\{0\}$, we have $\gamma w\in S_\delta$, where
$$
\gamma=\left(\int_{\Om}|w|^{1-\delta}f\,dx\right)^{-\frac{1}{1-\delta}}.
$$
Thus $w=\gamma^{-1} V_\delta$ or $-\gamma^{-1}V_\delta$. Since $V_\delta=\tau_\delta u_\delta$, the result follows. 
\qed
\end{enumerate}
\textbf{Proof of Theorem \ref{nsthm}:}
Suppose the inequality \eqref{inequality2} holds. Then, by Lemma \ref{nsbdd}, the sequence $\{u_n\}$ is uniformly bounded in $W_0^{1,p}(\Om)$. Now, proceeding with the exact arguments as in the proof of Theorem \ref{thm1}, the problem $\mathcal{(S)}$ admits a weak solution in $W_0^{1,p}(\Om)$. Conversely, let $u\in W_0^{1,p}(\Om)$ be a weak solution of the problem $\mathcal{(S)}$ and $\|\cdot\|$ be given by \eqref{norm}. Applying Lemma \ref{testfn} we choose $u$ as a test function in $\mathcal{(S)}$ and obtain
\begin{equation}\label{tst2}
\|u\|^p=\int_{\Om}u^{1-\delta}f\,dx.
\end{equation}
Again applying Lemma \ref{testfn}, first we choose $|v|\in W_0^{1,p}(\Om)$ as a test function in $\mathcal{(S)}$ and then using H\"older's inequality, we have
\begin{equation}\label{tst}
\int_{\Om}|v|u^{-\delta}f\,dx\leq C\|u\|^{p-1}\|v\|,
\end{equation}
for some positive constant $C$. Therefore, for every $v\in W_0^{1,p}(\Om)$, using \eqref{tst2} and \eqref{tst} along with H\"older's inequality, we obtain
\begin{multline*}
\int_{\Om}|v|^{1-\delta}f\,dx=\int_{\Om}\left(|v|u^{-\delta}f\right)^{1-\delta}\left(u^{1-\delta}f\right)^\delta\,dx\\
\leq\left(\int_{\Om}|v|u^{-\delta}f\,dx\right)^{1-\delta}\left(\int_{\Om}u^{1-\delta}f\,dx\right)^{\delta}
\leq C\|u\|^{p+\delta-1}\|v\|^{1-\delta},
\end{multline*}
which proves the inequality \eqref{inequality2}. \qed \\
\textbf{Proof of Theorem \ref{nsthm1}:} Noting \cite[Lemma $4.1$]{Canino} and then, proceeding along the lines of the proof of Theorem \ref{nsthm}, the result follows. \qed \\
\textbf{Proof of Theorem \ref{nsthm2}:} Noting \cite[Proposition $2.3$]{Caninononloc} and then, proceeding along the lines of the proof of Theorem \ref{nsthm}, the result follows. \qed
\section{Appendix}
In this section, we obtain the pointwise convergence of the gradient of the approximate solutions $\{u_n\}$ found in Lemma \ref{approx}. 
\begin{Theorem}\label{grad}\textbf{(Gradient convergence theorem in the mixed case)}
Let $1<p<\infty$. Suppose $\{u_n\}$ is the sequence of approximate solutions for the problem $(\mathcal{A})$ given by Lemma \ref{approx} and $u$ is the pointwise limit of $\{u_n\}$. For $0<\delta<1$, let $f\in L^m(\Om)\setminus\{0\}$ be nonnegative, where $m$ is given by \eqref{m} and for $\delta\geq 1$, let $f\in L^1(\Om)\setminus\{0\}$ be nonnegative respectively. Then, upto a subsequence, $\nabla u_n\to \nabla u$ pointwise almost everywhere in $\Om$.
\end{Theorem}
\begin{proof}
By the given hypothesis, for $0<\delta\leq 1$, by Lemma \ref{unibddless} and \ref{unibdequal}, 
\begin{equation}\label{gradlesseq}
\text{the sequence } \{u_n\} \text{ is uniformly bounded in } W_0^{1,p}(\Om).
\end{equation} 
If $\delta>1$, by Lemma \ref{unibddgrt}, the sequences
\begin{equation}\label{gradgrt}
\begin{split}
\{u_n\} \text{ and } \left\{u_n^\frac{\delta+p-1}{p}\right\} \text{ are uniformly bounded in } W^{1,p}_{\mathrm{loc}}(\Om) \text{ and } W_0^{1,p}(\Om) \text{ respectively.}
\end{split}
\end{equation}
Then, we have
\begin{equation}\label{consewk}
\begin{split}
u_n&\rightharpoonup u \text{ weakly in }W^{1,p}_{\mathrm{loc}}(\Om)\text{ and }
\end{split}
\end{equation}
\begin{equation}\label{consest}
\begin{split}
u_n&\to u \text{ strongly in }L^p_{\mathrm{loc}}(\Om).
\end{split}
\end{equation}
By Remark \ref{rmkapprox}, for all $n\in\mathbb{N}$, we obtain
\begin{align}\label{consegeq}
u&\geq u_n\text{ in }\mathbb{R}^N.
\end{align}
Also, we note that the function 
\begin{equation}\label{consemon}
|t|^{p-2}t\text{ is monotone increasing over }\mathbb{R}.
\end{equation}
Next, we prove the result in the following two steps.\\
\textbf{Step 1.} Let $K\subset\Omega$ be a compact set and consider a function $\phi_K\in C_c^{1}(\Omega)$ such that $\mathrm{supp}\,\phi_K=\omega$, $0\leq\phi_K\leq 1$ in $\Omega$ and $\phi_K\equiv 1$ in $K$. For $\mu>0$, we define the truncated function $T_{\mu}:\mathbb{R}\to\mathbb{R}$ by
\begin{equation}\label{trun}
T_{\mu}(s)=
\begin{cases}
s,\text{ if }|s|\leq\mu,\\
\mu\frac{s}{|s|},\text{ if }|s|>\mu.
\end{cases}
\end{equation}
Choosing $v_n=\phi_K T_{\mu}\big((u_n-u)\big)\in W_0^{1,p}(\Omega)$ as a test function in $(\mathcal{A})$, we obtain
\begin{equation}\label{gs1eqn1}
\begin{split}
I+J=R,
\end{split}
\end{equation}
where
$$
I=\int_{\Omega}|\nabla u_n|^{p-2}\nabla u_n\nabla v_n\,dx,
$$
$$
J=\int\limits_{\mathbb{R}^N}\int\limits_{\mathbb{R}^N}\mathcal{A}\big(u_n(x,y)\big)\big(v_n(x)-v_n(y)\big)\,d\mu\text{ and }R=\int_{\Omega}\frac{f_n}{\big(u_n+\frac{1}{n}\big)^\delta} v_n\,dx.
$$
\textbf{Estimate of $I$:} 
We have
\begin{equation}\label{estI}
\begin{split}
I&=\int_{\Omega}|\nabla u_n|^{p-2}\nabla u_n\nabla v_n\,dx\\
&=\int_{\Omega}\phi_K\big(|\nabla u_n|^{p-2}\nabla u_n-|\nabla u|^{p-2}\nabla u\big)\nabla T_{\mu}\big((u_n-u)\big)\,dx\\
&\quad+\int_{\Omega}\phi_K|\nabla u|^{p-2}\nabla u\nabla T_{\mu}\big((u_n-u)\big)\,dx+\int_{\Omega}T_{\mu}\big((u_n-u)\big)|\nabla u_n|^{p-2}\nabla u_n\nabla\phi_K\,dx\\
&=I_1+I_2+I_3.
\end{split}
\end{equation}
\textbf{Estimate of $I_2$:} Using \eqref{consewk} we get $T_{\mu}\big((u_n-u)\big)\rightharpoonup 0$ weakly in $W^{1,p}_{\mathrm{loc}}(\Omega)$. As a consequence,
\begin{equation}\label{estI_2}
\lim_{n\to\infty}I_2=\lim_{n\to\infty}\int_{\Omega}\phi_K|\nabla u|^{p-2}\nabla u\nabla T_{\mu}\big((u_n-u)\big)\,dx=0.
\end{equation}
\textbf{Estimate of $I_3$:} For $\omega=\mathrm{supp}\,\phi_K$, by H\"older's inequality, using the uniformly boundedness of $\{u_n\}$ in $W^{1,p}(\omega)$, we have
\begin{multline}\label{estI_3}
|I_3|=\left|\int_{\Omega}T_{\mu}\big((u_n-u)\big)|\nabla u_n|^{p-2}\nabla u_n\nabla\phi_K\,dx\right|
\leq\int_{\Omega}|\nabla\phi_K\|u_n-u\|\nabla u_n|^{p-1}\,dx\\
\leq\Vert\nabla\phi_K\Vert_{L^\infty(\omega)}\left(\int_{\omega}|\nabla u_n|^p\,dx\right)^\frac{p-1}{p}\Vert u_n-u\Vert_{L^p(\omega)}
\leq C\Vert u_n-u\Vert_{L^p(\omega)},
\end{multline}
for some constant $C>0$, independent of $n$. Thus using \eqref{consest} in \eqref{estI_3}, we obtain
\begin{equation}\label{estI_3final}
\lim_{n\to\infty}I_3=\lim_{n\to\infty}\int_{\Omega}T_{\mu}\big((u_n-u)\big)|\nabla u_n|^{p-2}\nabla u_n\nabla\phi_K\,dx=0.
\end{equation}
Therefore, using the estimates \eqref{estI_2} and \eqref{estI_3final} in \eqref{estI} we obtain
\begin{equation}\label{estIfinal}
\limsup_{n\to\infty}I=\limsup_{n\to\infty}I_1=\limsup_{n\to\infty}\int_{\Omega}\phi_K\big(|\nabla u_n|^{p-2}\nabla u_n-|\nabla u|^{p-2}\nabla u\big)\nabla T_{\mu}\big((u_n-u)\big)\,dx.
\end{equation}
\textbf{Estimate of $J$:} We have
\begin{equation}\label{estJ}
\begin{split}
J&=\int\limits_{\mathbb{R}^N}\int\limits_{\mathbb{R}^N}\mathcal{A}\big(u_n(x,y)\big)\big(v_n(x)-v_n(y)\big)\,d\mu\\
&=\int\limits_{\mathbb{R}^N}\int\limits_{\mathbb{R}^N}\phi_K(x)\Big(\mathcal{A}\big(u_n(x,y)\big)-\mathcal{A}\big(u(x,y)\big)\Big)\Big(T_{\mu}\big((u_n-u)(x)\big)-T_{\mu}\big((u_n-u)(y)\big)\Big)\,d\mu\\
&\quad+\int\limits_{\mathbb{R}^N}\int\limits_{\mathbb{R}^N}T_{\mu}\big((u_n-u)(y)\big)\mathcal{A}\big(u_n(x,y)\big)\big(\phi_K(x)-\phi_K(y)\big)\,d\mu\\
&\quad\quad+\int\limits_{\mathbb{R}^N}
\int\limits_{\mathbb{R}^N} T_{\mu}\big((u_n-u)(y)\big)\mathcal{A}\big( u(x,y)\big)\big(\phi_K(y)-\phi_K(x)\big)\,d\mu\\
&\quad\quad\quad+\int\limits_{\mathbb{R}^N}\int\limits_{\mathbb{R}^N}\mathcal{A}\big(u(x,y)\big)\Big(\phi_K(x)T_{\mu}\big((u_n-u)(x)\big)-\phi_K(y)T_{\mu}\big((u_n-u)(y)\big)\Big)\,d\mu\\
&=J_1+J_2+J_3+J_4.
\end{split}
\end{equation}
\textbf{Estimate of $J_1$:} We claim that 
$$
J_1=\int\limits_{\mathbb{R}^N}\int\limits_{\mathbb{R}^N}\phi_K(x)\Big(\mathcal{A}\big(u_n(x,y)\big)-\mathcal{A}\big(u(x,y)\big)\Big)\Big(T_{\mu}\big((u_n-u)(x)\big)-T_{\mu}\big((u_n-u)(y)\big)\Big)\,d\mu\geq 0.
$$
To this end, it is enough to prove that, for almost every $(x,y)\in\mathbb{R}^N\times\mathbb{R}^N$, we have 
\begin{align*}
\hat{J}_{1}&=\phi_K(x)\Big(\mathcal{A}\big(u_n(x,y)\big)-\mathcal{A}\big(u(x,y)\big)\Big)\Big(T_{\mu}\big((u_n-u)(x)\big)-T_{\mu}\big((u_n-u)(y)\big)\Big)\geq 0.
\end{align*}
We observe that
$$
\mathbb{R}^N\times\mathbb{R}^N=\cup_{i=1}^{4}S_i,
$$
where
$$
S_1=\Big\{(x,y)\in\mathbb{R}^N\times\mathbb{R}^N:|(u_n-u)(x)|\leq\mu,\, |(u_n-u)(y)|\leq\mu\Big\},
$$

$$
S_2=\Big\{(x,y)\in\mathbb{R}^N\times\mathbb{R}^N:|(u_n-u)(x)|\leq\mu<|(u_n-u)(y)|\Big\},
$$

$$
S_3=\Big\{(x,y)\in\mathbb{R}^N\times\mathbb{R}^N:|(u_n-u)(y)|\leq\mu< |(u_n-u)(x)|\Big\}
$$
and
$$
S_4=\Big\{(x,y)\in\mathbb{R}^N\times\mathbb{R}^N:|(u_n-u)(x)|>\mu,\, |(u_n-u)(y)|>\mu\Big\}.
$$
\begin{enumerate}
\item[Case $1$.] Let $x,y\in S_1$. Then, $|(u_n-u)(x)|\leq\mu$ and $|(u_n-u)(y)|\leq\mu$, which gives 
$$
T_{\mu}\big((u_n-u)(x)\big)=(u_n-u)(x) \text{ and }T_{\mu}\big((u_n-u)(y)\big)=(u_n-u)(y).
$$
Therefore, we have
$$
\hat{J}_{1}=\phi_K(x)\big(\mathcal{A}(u_n(x,y))-\mathcal{A}(u(x,y))\big)((u_n-u)(x)-(u_n-u)(y)),
$$
which is nonnegative, by Lemma \ref{AI}.

\item[Case $2$.] Let $x,y\in S_2$. Then, $|(u_n-u)(x)|\leq\mu<|(u_n-u)(y)|$, which gives from \eqref{consegeq} that $u(x)-u_n(x)\leq\mu<u(y)-u_n(y)$. Therefore, we have
\begin{equation}\label{c2J_1}
T_{\mu}\big((u_n-u)(x)\big)-T_{\mu}\big((u_n-u)(y)\big)=(u_n-u)(x)+\mu\geq 0.
\end{equation}
Moreover, in this case, $u_n(x)-u_n(y)>u(x)-u(y)$. 
Thus, by \eqref{consemon}, we obtain
\begin{equation}\label{c21J_1}
\mathcal{A}(u_n(x,y))-\mathcal{A}(u(x,y))\geq 0.
\end{equation}
From \eqref{c2J_1} and \eqref{c21J_1}, it follows that $\hat{J}_1\geq 0$.

\item[Case $3$.] Let $x,y\in S_3$. Then, $|(u_n-u)(y)|\leq\mu<|(u_n-u)(x)|$, which gives from \eqref{consegeq} that $u(y)-u_n(y)\leq\mu<u(x)-u_n(x)$. Therefore, we have
\begin{equation}\label{c3J_1}
T_{\mu}\big((u_n-u)(x)\big)-T_{\mu}\big((u_n-u)(y)\big)=-\mu-(u_n-u)(y)\leq 0.
\end{equation}
Moreover, in this case, we get $u_n(x)-u_n(y)<u(x)-u(y)$. Thus, by \eqref{consemon}, we deduce
\begin{equation}\label{c31J_1}
\mathcal{A}(u_n(x,y))-\mathcal{A}(u(x,y))\leq 0.
\end{equation}
From \eqref{c3J_1} and \eqref{c31J_1}, it follows that $\hat{J}_1\geq 0$.

\item[Case $4$.] Let $x,y\in S_4$. Then, $|(u_n-u)(x)|>\mu$ and $|(u_n-u)(y)|>\mu$, which gives
\begin{equation}\label{c4J_1}
T_{\mu}\big((u_n-u)(x)\big)-T_{\mu}\big((u_n-u)(y)\big)=-\mu+\mu=0.
\end{equation}
Hence, $\hat{J}_1=0$.
\end{enumerate}
Therefore, we have
\begin{equation}\label{estJ_1}
J_1\geq 0.
\end{equation}
Hence, the claim follows.\\
\textbf{Estimate of $J_2$:}
We have 
\begin{align*}
J_2&=\int\limits_{\mathbb{R}^N}\int\limits_{\mathbb{R}^N}T_{\mu}\big((u_n-u)(y)\big)\mathcal{A}\big(u_n(x,y)\big)\big(\phi_K(x)-\phi_K(y)\big)\,d\mu.
\end{align*}
First, we claim that for every $\epsilon>0$, there exists a compact set $\mathcal{K}\subset\mathbb{R}^{2N}$, such that 
\begin{equation}\label{estJ21}
\begin{split}
J_2^{1}&=\int_{\mathbb{R}^{2N}\setminus\mathcal{K}}T_{\mu}\big((u_n-u)(y)\big)\mathcal{A}\big(u_n(x,y)\big)\big(\phi_K(x)-\phi_K(y)\big)\,d\mu\leq\frac{\epsilon}{2},
\end{split}
\end{equation}
for all $n\in\mathbb{N}$. To this end, let us set
$$
\omega=\mathrm{supp}\,\phi_K,\quad\quad\omega'=\mathbb{R}^{2N}\setminus(\omega^c\times\omega^c),
$$
and as $\phi_K\in C_c^{1}(\Om)$, for every $\epsilon>0$, there exists a compact set $\mathcal{L}=\mathcal{L}(\epsilon)\subset\mathbb{R}^{2N}$ such that
\begin{equation}\label{J21app}
\left(\int_{\mathbb{R}^{2N}\setminus\mathcal{L}}\frac{|\phi_K(x)-\phi_K(y)|^p}{|x-y|^{N+ps}}\,dx dy\right)^\frac{1}{p}\leq\frac{\epsilon}{2}.
\end{equation}
Then we choose $\mathcal{K}=\mathcal{L}$ and applying H\"older's inequality, we obtain
\begin{multline}\label{estJ211}
J_2^{1}=\int_{\mathbb{R}^{2N}\setminus\mathcal{K}}T_{\mu}\big((u_n-u)(y)\big)\mathcal{A}\big(u_n(x,y)\big)\big(\phi_K(x)-\phi_K(y)\big)\,d\mu\\
\leq\left(\int_{\omega'\setminus\mathcal{K}}\frac{|u_n(x)-u_n(y)|^p}{|x-y|^{N+ps}}\,dx dy\right)^\frac{p-1}{p}
\left(\int_{\omega'\setminus\mathcal{K}}\frac{|T_{\mu}\big((u_n-u)(y)\big)|^p|\phi_K(x)-\phi_K(y)|^p}{|x-y|^{N+ps}}\,dx dy\right)^\frac{1}{p}\\
\leq\mu\left(\int_{\omega'\setminus\mathcal{K}}\frac{|u_n(x)-u_n(y)|^p}{|x-y|^{N+ps}}\,dx dy\right)^\frac{p-1}{p}\left(\int_{\omega'\setminus\mathcal{K}}\frac{|\phi_K(x)-\phi_K(y)|^p}{|x-y|^{N+ps}}\,dx dy\right)^\frac{1}{p}.
\end{multline}
If $0<\delta\leq 1$, then using the estimate \eqref{J21app} and from \eqref{gradlesseq} the uniform boundedness of $\{u_n\}$ over $W_0^{1,p}(\Omega)$ in \eqref{estJ211} the claim \eqref{estJ21} follows. If $\delta>1$, first we note that by Lemma \ref{approx}, there exists a constant $C(\omega)>0$, independent of $n$ such that $u_n\geq C(\omega)$ in $\omega$. Choosing $q=\frac{\delta+p-1}{p}>1$ in [\eqref{cneqn}, Lemma \ref{Cnlemma}], we obtain
\begin{equation}\label{estJ212}
\frac{|u_n(x)-u_n(y)|^p}{|x-y|^{N+ps}}\leq C(\omega)^{1-\delta}\frac{\left|u_n(x)^\frac{\delta+p-1}{p}-u_n(y)^\frac{\delta+p-1}{p}\right|^p}{|x-y|^{N+ps}},
\end{equation}
for almost every $(x,y)\in\omega'$. Hence, using \eqref{estJ212} in \eqref{estJ211}, we get
\begin{equation}\label{estJ213}
\begin{split}
J_2^{1}&\leq\mu\left(\int_{\mathbb{R}^{2N}}\frac{\left|u_n(x)^\frac{\delta+p-1}{p}-u_n(y)^\frac{\delta+p-1}{p}\right|^p}{|x-y|^{N+ps}}\,dx dy\right)^\frac{p-1}{p}\left(\int_{\mathbb{R}^{2N}\setminus\mathcal{K}}\frac{|\phi_K(x)-\phi_K(y)|^p}{|x-y|^{N+ps}}\,dx dy\right)^\frac{1}{p}\leq \frac{\epsilon}{2},
\end{split}
\end{equation}
where in the last line above, we have used the estimate \eqref{J21app} and the uniform boundedness of $\left\{u_n^\frac{\delta+p-1}{p}\right\}$ in $W_0^{1,p}(\Om)$ that follows from \eqref{gradgrt}. Hence the claim \eqref{estJ21} follows. On the other hand, let $E\subset\mathcal{K}$ be an arbitrary measurable set. Then, proceeding analogous to \eqref{estJ211} and \eqref{estJ212}, for some uniform positive constant $C$, we have
\begin{equation}\label{estJ22n}
\begin{split}
J_2^{2}&=\int_{E}T_{\mu}\big((u_n-u)(y)\big)\mathcal{A}\big(u_n(x,y)\big)\big(\phi_K(x)-\phi_K(y)\big)\,d\mu\\
&\leq C\left(\int_{E}\frac{|\phi_K(x)-\phi_K(y)|^p}{|x-y|^{N+ps}}\,dx dy\right)^\frac{1}{p}.
\end{split}
\end{equation}
Thus, if $|E|\to 0$, we have $J_2^{2}\to 0$, uniformly in $n$. Note that
$$
T_{\mu}\big((u_n-u)(y)\big)\mathcal{A}\big(u_n(x,y)\big)\big(\phi_K(x)-\phi_K(y)\big)\to 0\text{ almost everywhere in }\mathbb{R}^{2N}.
$$
Then, by Vitali's theorem for any $\epsilon>0$, there exists $n_0$ such that, if $n\geq n_0$, 
\begin{equation}\label{J22lim}
\int_{\mathcal{K}}T_{\mu}\big((u_n-u)(y)\big)\mathcal{A}\big(u_n(x,y)\big)\big(\phi_K(x)-\phi_K(y)\big)\,d\mu\leq\frac{\epsilon}{2}.
\end{equation}
Therefore, from \eqref{estJ21} and \eqref{J22lim}, we obtain
\begin{equation}\label{estJ_2}
\lim_{n\to\infty}J_2=0.
\end{equation}
\textbf{Estimate of $J_3$:} Recalling from \eqref{gradlesseq} that for $0<\delta\leq 1$, $u\in W_0^{1,p}(\Om)$ and by \eqref{gradgrt} for $\delta>1$, $u\in W^{1,p}_{\mathrm{loc}}(\Omega)$ such that $u^\frac{\delta+p-1}{p}\in W_0^{1,p}(\Omega)$, we can proceed with exactly the same steps followed to estimate $J_2$ above to obtain
\begin{equation}\label{estJ_3}
\lim_{n\to\infty}J_3=\lim_{n\to\infty}\int\limits_{\mathbb{R}^N}
\int\limits_{\mathbb{R}^N} T_{\mu}\big((u_n-u)(y)\big)\mathcal{A}\big( u(x,y)\big)\big(\phi_K(y)-\phi_K(x)\big)\,d\mu=0.
\end{equation}
\textbf{Estimate of $J_4$:}
For $v_n=\phi_K T_{\mu}\big((u_n-u)\big)$, we have 
\begin{align*}
J_4&=\int\limits_{\mathbb{R}^N}\int\limits_{\mathbb{R}^N}\mathcal{A}\big(u(x,y)\big)\big(v_n(x)-v_n(y)\big)\,d\mu.
\end{align*}
Firstly, we claim that for every $\epsilon>0$, there exists a compact set $\mathcal{K}\subset\mathbb{R}^{2N}$, such that 
\begin{equation}\label{estJ244}
\begin{split}
J_4^{1}&=\int_{\mathbb{R}^{2N}\setminus\mathcal{K}}\mathcal{A}\big(u(x,y)\big)\big(v_n(x)-v_n(y)\big)\,d\mu\leq\frac{\epsilon}{2},
\end{split}
\end{equation}
for all $n\in\mathbb{N}$. To this end, let us set
$$
\omega=\mathrm{supp}\,\phi_K,\quad\quad\omega'=\mathbb{R}^{2N}\setminus(\omega^c\times\omega^c).
$$
By Lemma \ref{approx}, there exists $C(\omega)>0$, independent of $n$ such that $u_n\geq C(\omega)$ in $\omega$. Thus, for $\delta>1$, again applying [\eqref{cneqn}, Lemma \ref{Cnlemma}] for $q=\frac{\delta+p-1}{p}>1$, we obtain
\begin{equation}\label{estJ214}
\frac{|u(x)-u(y)|^p}{|x-y|^{N+ps}}\leq C(\omega)^{1-\delta}\frac{\left|u(x)^\frac{\delta+p-1}{p}-u(y)^\frac{\delta+p-1}{p}\right|^p}{|x-y|^{N+ps}},
\end{equation}
for almost every $(x,y)\in\omega'$. Therefore, using \eqref{estJ214} and from \eqref{gradgrt}, using the fact $u^\frac{\delta+p-1}{p}\in W_0^{1,p}(\Om)$, we have for every $\epsilon>0$, there exists a compact set $\mathcal{K}=\mathcal{K}(\epsilon)\subset\mathbb{R}^{2N}$ such that
\begin{equation}\label{J24app}
\left(\int\limits_{\mathbb{R}^{2N}\setminus\mathcal{K}}\frac{|u(x)-u(y)|^p}{|x-y|^{N+ps}}\,dx dy\right)^\frac{1}{p}
\leq C(\omega)^{1-\delta}\left(\int\limits_{\mathbb{R}^{2N}\setminus\mathcal{K}}\frac{|u(x)^\frac{\delta+p-1}{p}-u(y)^\frac{\delta+p-1}{p}|^p}{|x-y|^{N+ps}}\,dx dy\right)^\frac{1}{p}\leq\frac{\epsilon}{2}.
\end{equation}
Using the estimate \eqref{J24app}, the uniform boundedness of $v_n$ in $W_0^{1,p}(\Om)$ and H\"older's inequality, we obtain
\begin{multline}\label{estJ2124}
J_4^{1}=\int_{\mathbb{R}^{2N}\setminus\mathcal{K}}\mathcal{A}\big(u(x,y)\big)\big(v_n(x)-v_n(y)\big)\,d\mu\\
\leq\left(\int_{\omega'\setminus\mathcal{K}}\frac{|u(x)-u(y)|^p}{|x-y|^{N+ps}}\,dx dy\right)^\frac{p-1}{p}\left(\int_{\omega'\setminus\mathcal{K}}\frac{|v_n(x)-v_n(y)|^p}{|x-y|^{N+ps}}\,dx dy\right)^\frac{1}{p}
\leq\frac{\epsilon}{2},
\end{multline}
which proves \eqref{estJ244}. For $0<\delta\leq 1$, noting \eqref{gradlesseq} and arguing similarly as above, we obtain \eqref{estJ244}. On the other hand, let $E\subset\mathcal{K}$ be an arbitrary measurable set. Then, proceeding analogously to \eqref{estJ2124}, for some uniform positive constant $C$,
\begin{equation}\label{estJ22}
J_4^{2}=\int_{E}\mathcal{A}\big(u(x,y)\big)\big(v_n(x)-v_n(y)\big)\,d\mu
\leq C\left(\int_{E}\frac{|u(x)-u(y)|^p}{|x-y|^{N+ps}}\,dx dy\right)^\frac{1}{p}.
\end{equation}
Thus, if $|E|\to 0$,we have $J_4^{2}\to 0$, uniformly in $n$. Note that
$$
\mathcal{A}\big(u(x,y)\big)\big(v_n(x)-v_n(y)\big)\to 0\text{ almost everywhere in }\mathbb{R}^{2N}.
$$
Then, by Vitali's theorem for any $\epsilon>0$, there exists $n_0$ such that, if $n\geq n_0$, 
\begin{equation}\label{J42lim}
\int_{\mathcal{K}}\mathcal{A}\big(u(x,y)\big)\big(v_n(x)-v_n(y)\big)\,d\mu\leq\frac{\epsilon}{2}.
\end{equation}
Therefore, from \eqref{estJ244} and \eqref{J42lim}, we obtain
\begin{equation}\label{estJ4final}
\lim_{n\to\infty}J_4=0.
\end{equation}
Now, employing the estimates \eqref{estJ_1}, \eqref{estJ_2}, \eqref{estJ_3} and \eqref{estJ4final} in \eqref{estJ} we obtain
\begin{equation}\label{estJfinal}
\lim_{n\to\infty}J\geq 0.
\end{equation}
\textbf{Estimate of $R$:} Recalling that $\mathrm{supp}\,\phi_K=\omega$, by Lemma \ref{approx}, there exists a constant $C(\omega)>0$, independent of $n$, such that $u\geq C(\omega)>0$ in $\Om$. Hence, we have
\begin{equation}\label{estK}
R=\int_{\Om}\frac{f_n}{\big(u_n+\frac{1}{n}\big)^\delta}v_n\,dx
\leq\frac{\|f\|_{L^1(\Om)}}{C(\omega)^\delta}\mu.
\end{equation}
Therefore, for every fixed $\mu>0$, using the estimates \eqref{estIfinal}, \eqref{estJfinal} and \eqref{estK} in \eqref{gs1eqn1}, we obtain
\begin{equation}\label{gest}
\begin{split}
\lim\sup_{n\to\infty}\int_{K}\big(|\nabla u_n|^{p-2}\nabla u_n-|\nabla u|^{p-2}\nabla u\big)\nabla T_{\mu}\big((u_n-u)\big)\,dx\leq C\mu,
\end{split}
\end{equation}
for some constant $C=C\big(\omega,\|f\|_{L^1(\Om)}\big)>0$.\\
\textbf{Step $2$.} Let us define the function
\begin{equation}\label{en}
e_n(x)=\big(|\nabla u_n|^{p-2}\nabla u_n-|\nabla u|^{p-2}\nabla u\big)\nabla(u_n-u).
\end{equation}
Note that by Lemma \ref{AI}, we have $e_n\geq 0$ in $\Omega$. We divide the compact set $K$ by
$$
E_n^{\mu}=\big\{x\in K:|(u_n-u)(x)|\leq\mu\big\},\,F_n^{\mu}=\big\{x\in K:|(u_n-u)(x)|>\mu\big\}.
$$
Let $\gamma\in(0,1)$ be fixed. Then, from H\"older's inequality,
\begin{equation}\label{enest}
\int_{K}e_n^{\gamma}\,dx=\int_{E_n^{\mu}}e_n^{\gamma}\,dx+\int_{F_n^{\mu}}e_n^{\gamma}\,dx
\leq \left(\int_{E_n^{\mu}}e_n\,dx\right)^\gamma|E_n^{\mu}|^{1-\gamma}+\left(\int_{F_n^{\mu}}e_n\,dx\right)^\gamma|F_n^{\mu}|^{1-\gamma}.
\end{equation}
Now, since $\{u_n\}$ is uniformly bounded in $W^{1,p}(K)$, the sequence $\{e_n\}$ is uniformly bounded in $L^1(K)$. Furthermore, $\lim_{n\to\infty}|F_n^{\mu}|=0$. Hence, from \eqref{gest} and \eqref{enest}, we have
\begin{equation}\label{enfinalest}
\lim\sup_{n\to\infty}\int_{K}e_n^{\gamma}\,dx\leq\lim\sup_{n\to\infty}\left(\int_{E_n^{\mu}}e_n\,dx\right)^\gamma|E_n^{\mu}|^{1-\gamma}
\leq\big(C\mu\big)^\gamma|\Omega|^{1-\gamma}.  
\end{equation}
Letting $\mu\to 0$ in \eqref{enfinalest}, the sequence $\Big\{e_n^{\gamma}\Big\}$ converges to $0$ strongly in $L^1(K)$. Therefore, using a sequence of compact sets $K$, upto a subsequence
$$
e_n(x)\to 0\text{ almost everywhere in }\Omega,
$$
which along with Lemma \ref{AI} gives
$$
\nabla u_n(x)\to\nabla u(x)\text{ for almost every }x\in\Omega.
$$
Hence, the result follows.
\end{proof}


\begin{thebibliography}{10}

\bibitem{Adi}
Adimurthi, Jacques Giacomoni, and Sanjiban Santra.
\newblock Positive solutions to a fractional equation with singular
  nonlinearity.
\newblock {\em J. Differential Equations}, 265(4):1191--1226, 2018.

\bibitem{Alves}
C.~O. Alves, J.~V. Goncalves, and L.~A. Maia.
\newblock Singular nonlinear elliptic equations in {${\bf R}^N$}.
\newblock {\em Abstr. Appl. Anal.}, 3(3-4):411--423, 1998.

\bibitem{GFA}
Giovanni Anello, Francesca Faraci, and Antonio Iannizzotto.
\newblock On a problem of {H}uang concerning best constants in {S}obolev
  embeddings.
\newblock {\em Ann. Mat. Pura Appl. (4)}, 194(3):767--779, 2015.

\bibitem{arcoya}
David Arcoya and Lucio Boccardo.
\newblock Multiplicity of solutions for a {D}irichlet problem with a singular
  and a supercritical nonlinearities.
\newblock {\em Differential Integral Equations}, 26(1-2):119--128, 2013.

\bibitem{AM}
David Arcoya and Lourdes Moreno-M\'{e}rida.
\newblock Multiplicity of solutions for a {D}irichlet problem with a strongly
  singular nonlinearity.
\newblock {\em Nonlinear Anal.}, 95:281--291, 2014.

\bibitem{Aubin}
Thierry Aubin.
\newblock Probl\`emes isop\'{e}rim\'{e}triques et espaces de {S}obolev.
\newblock {\em J. Differential Geometry}, 11(4):573--598, 1976.

\bibitem{BG}
Kaushik Bal and Prashanta Garain.
\newblock Multiplicity of solution for a quasilinear equation with singular
  nonlinearity.
\newblock {\em Mediterr. J. Math.}, 17(3):Paper No. 91, 20, 2020.

\bibitem{BGmm}
Kaushik Bal and Prashanta Garain.
\newblock Weighted and anisotropic sobolev inequality with extremal.
\newblock {\em Manuscripta Mathematica}, pages 1--17, 2021.

\bibitem{BBCK}
Martin~T. Barlow, Richard~F. Bass, Zhen-Qing Chen, and Moritz Kassmann.
\newblock Non-local {D}irichlet forms and symmetric jump processes.
\newblock {\em Trans. Amer. Math. Soc.}, 361(4):1963--1999, 2009.

\bibitem{BDMP}
Bego\~{n}a Barrios, Ida De~Bonis, Mar\'{\i}a Medina, and Ireneo Peral.
\newblock Semilinear problems for the fractional laplacian with a singular
  nonlinearity.
\newblock {\em Open Math.}, 13(1):390--407, 2015.

\bibitem{BK}
M.~Belloni and B.~Kawohl.
\newblock A direct uniqueness proof for equations involving the {$p$}-{L}aplace
  operator.
\newblock {\em Manuscripta Math.}, 109(2):229--231, 2002.

\bibitem{BSVV2}
Stefano {Biagi}, Serena {Dipierro}, Enrico {Valdinoci}, and Eugenio {Vecchi}.
\newblock {Mixed local and nonlocal elliptic operators: regularity and maximum
  principles}.
\newblock {\em arXiv e-prints}, page arXiv:2005.06907, May 2020.

\bibitem{BSVV1}
Stefano {Biagi}, Serena {Dipierro}, Enrico {Valdinoci}, and Eugenio {Vecchi}.
\newblock {Semilinear elliptic equations involving mixed local and nonlocal
  operators}.
\newblock {\em arXiv e-prints}, page arXiv:2006.05830, June 2020.

\bibitem{Biagi2}
Stefano {Biagi}, Serena {Dipierro}, Enrico {Valdinoci}, and Eugenio {Vecchi}.
\newblock {A Faber-Krahn inequality for mixed local and nonlocal operators}.
\newblock {\em arXiv e-prints}, page arXiv:2104.00830, April 2021.

\bibitem{Biagi1}
Stefano {Biagi}, Dimitri {Mugnai}, and Eugenio {Vecchi}.
\newblock {Global boundedness and maximum principle for a Brezis-Oswald
  approach to mixed local and nonlocal operators}.
\newblock {\em arXiv e-prints}, page arXiv:2103.11382, March 2021.

\bibitem{BocMu}
Lucio Boccardo and Fran\c{c}ois Murat.
\newblock Almost everywhere convergence of the gradients of solutions to
  elliptic and parabolic equations.
\newblock {\em Nonlinear Anal.}, 19(6):581--597, 1992.

\bibitem{BocOr}
Lucio Boccardo and Luigi Orsina.
\newblock Semilinear elliptic equations with singular nonlinearities.
\newblock {\em Calc. Var. Partial Differential Equations}, 37(3-4):363--380,
  2010.

\bibitem{BLP}
L.~Brasco, E.~Lindgren, and E.~Parini.
\newblock The fractional {C}heeger problem.
\newblock {\em Interfaces Free Bound.}, 16(3):419--458, 2014.

\bibitem{BrPr}
Lorenzo Brasco and Enea Parini.
\newblock The second eigenvalue of the fractional {$p$}-{L}aplacian.
\newblock {\em Adv. Calc. Var.}, 9(4):323--355, 2016.

\bibitem{Silva}
S.~{Buccheri}, J.~V. {da Silva}, and L.~H. {de Miranda}.
\newblock {A System of Local/Nonlocal $p$-Laplacians: The Eigenvalue Problem
  and Its Asymptotic Limit as $p\to\infty$}.
\newblock {\em arXiv e-prints}, page arXiv:2001.05985, January 2020.

\bibitem{CaninoSci}
A.~Canino, M.~Grandinetti, and B.~Sciunzi.
\newblock Symmetry of solutions of some semilinear elliptic equations with
  singular nonlinearities.
\newblock {\em J. Differential Equations}, 255(12):4437--4447, 2013.

\bibitem{Caninovar}
Annamaria Canino and Marco Degiovanni.
\newblock A variational approach to a class of singular semilinear elliptic
  equations.
\newblock {\em J. Convex Anal.}, 11(1):147--162, 2004.

\bibitem{Caninononloc}
Annamaria Canino, Luigi Montoro, Berardino Sciunzi, and Marco Squassina.
\newblock Nonlocal problems with singular nonlinearity.
\newblock {\em Bull. Sci. Math.}, 141(3):223--250, 2017.

\bibitem{Caninouni}
Annamaria Canino and Berardino Sciunzi.
\newblock A uniqueness result for some singular semilinear elliptic equations.
\newblock {\em Commun. Contemp. Math.}, 18(6):1550084, 9, 2016.

\bibitem{Canino}
Annamaria Canino, Berardino Sciunzi, and Alessandro Trombetta.
\newblock Existence and uniqueness for {$p$}-{L}aplace equations involving
  singular nonlinearities.
\newblock {\em NoDEA Nonlinear Differential Equations Appl.}, 23(2):Art. 8, 18,
  2016.

\bibitem{Song2}
Zhen-Qing Chen, Panki Kim, and Renming Song.
\newblock Global heat kernel estimates for {$\Delta+\Delta^{\alpha/2}$} in
  half-space-like domains.
\newblock {\em Electron. J. Probab.}, 17:no. 32, 32, 2012.

\bibitem{Song1}
Zhen-Qing Chen, Panki Kim, Renming Song, and Zoran Vondra\v{c}ek.
\newblock Boundary {H}arnack principle for {$\Delta+\Delta^{\alpha/2}$}.
\newblock {\em Trans. Amer. Math. Soc.}, 364(8):4169--4205, 2012.

\bibitem{CRT}
M.~G. Crandall, P.~H. Rabinowitz, and L.~Tartar.
\newblock On a {D}irichlet problem with a singular nonlinearity.
\newblock {\em Comm. Partial Differential Equations}, 2(2):193--222, 1977.

\bibitem{Silvas}
Jo\~{a}o~Vitor da~Silva and Ariel~M. Salort.
\newblock A limiting problem for local/non-local {$p$}-{L}aplacians with
  concave-convex nonlinearities.
\newblock {\em Z. Angew. Math. Phys.}, 71(6):Paper No. 191, 27, 2020.

\bibitem{Dama}
Lucio Damascelli.
\newblock Comparison theorems for some quasilinear degenerate elliptic
  operators and applications to symmetry and monotonicity results.
\newblock {\em Ann. Inst. H. Poincar\'{e} Anal. Non Lin\'{e}aire},
  15(4):493--516, 1998.

\bibitem{DeCave}
Linda~Maria De~Cave.
\newblock Nonlinear elliptic equations with singular nonlinearities.
\newblock {\em Asymptot. Anal.}, 84(3-4):181--195, 2013.

\bibitem{DeCave2}
Linda~Maria De~Cave, Riccardo Durastanti, and Francescantonio Oliva.
\newblock Existence and uniqueness results for possibly singular nonlinear
  elliptic equations with measure data.
\newblock {\em NoDEA Nonlinear Differential Equations Appl.}, 25(3):Paper No.
  18, 35, 2018.

\bibitem{Rossi1}
Leandro~M. Del~Pezzo, Ra\'{u}l Ferreira, and Julio~D. Rossi.
\newblock Eigenvalues for a combination between local and nonlocal
  {$p$}-{L}aplacians.
\newblock {\em Fract. Calc. Appl. Anal.}, 22(5):1414--1436, 2019.

\bibitem{Hitchhiker'sguide}
Eleonora Di~Nezza, Giampiero Palatucci, and Enrico Valdinoci.
\newblock Hitchhiker's guide to the fractional {S}obolev spaces.
\newblock {\em Bull. Sci. Math.}, 136(5):521--573, 2012.

\bibitem{DPV20}
Serena {Dipierro}, Edoardo {Proietti Lippi}, and Enrico {Valdinoci}.
\newblock {Linear theory for a mixed operator with Neumann conditions}.
\newblock {\em arXiv e-prints}, page arXiv:2006.03850, June 2020.

\bibitem{DPV21}
Serena {Dipierro}, Edoardo {Proietti Lippi}, and Enrico {Valdinoci}.
\newblock {(Non)local logistic equations with Neumann conditions}.
\newblock {\em arXiv e-prints}, page arXiv:2101.02315, January 2021.

\bibitem{DRXJV20}
Serena {Dipierro}, Xavier {Ros-Oton}, Joaquim {Serra}, and Enrico {Valdinoci}.
\newblock {Non-symmetric stable operators: regularity theory and integration by
  parts}.
\newblock {\em arXiv e-prints}, page arXiv:2012.04833, December 2020.

\bibitem{EP1}
G.~Ercole and G.~A. Pereira.
\newblock Fractional {S}obolev inequalities associated with singular problems.
\newblock {\em Math. Nachr.}, 291(11-12):1666--1685, 2018.

\bibitem{EPS}
G.~Ercole, G.~A. Pereira, and R.~Sanchis.
\newblock Asymptotic behavior of extremals for fractional {S}obolev
  inequalities associated with singular problems.
\newblock {\em Ann. Mat. Pura Appl. (4)}, 198(6):2059--2079, 2019.

\bibitem{EP}
Grey Ercole and Gilberto de~Assis Pereira.
\newblock On a singular minimizing problem.
\newblock {\em J. Anal. Math.}, 135(2):575--598, 2018.

\bibitem{Evans}
Lawrence~C. Evans.
\newblock {\em Partial differential equations}, volume~19 of {\em Graduate
  Studies in Mathematics}.
\newblock American Mathematical Society, Providence, RI, 1998.

\bibitem{Fang}
Yanqin {Fang}.
\newblock {Existence, Uniqueness of Positive Solution to a Fractional
  Laplacians with Singular Nonlinearity}.
\newblock {\em arXiv e-prints}, page arXiv:1403.3149, March 2014.

\bibitem{FL}
Giovanni Franzina and Pier~Domenico Lamberti.
\newblock Existence and uniqueness for a {$p$}-{L}aplacian nonlinear eigenvalue
  problem.
\newblock {\em Electron. J. Differential Equations}, pages No. 26, 10, 2010.

\bibitem{FP}
Giovanni Franzina and Giampiero Palatucci.
\newblock Fractional {$p$}-eigenvalues.
\newblock {\em Riv. Math. Univ. Parma (N.S.)}, 5(2):373--386, 2014.

\bibitem{G}
Prashanta {Garain}.
\newblock {On a degenerate singular elliptic problem}.
\newblock {\em (To appear in Mathematische Nachrichten)}, page
  arXiv:1803.02102, March 2018.

\bibitem{GK}
Prashanta {Garain} and Juha {Kinnunen}.
\newblock {On the regularity theory for mixed local and nonlocal quasilinear
  elliptic equations}.
\newblock {\em arXiv e-prints}, page arXiv:2102.13365, February 2021.

\bibitem{GK1}
Prashanta {Garain} and Juha {Kinnunen}.
\newblock {Weak Harnack inequality for a mixed local and nonlocal parabolic
  equation}.
\newblock {\em arXiv e-prints}, page arXiv:2105.15016, May 2021.

\bibitem{GMwgt}
Prashanta Garain and Tuhina Mukherjee.
\newblock On a class of weighted {$p$}-{L}aplace equation with singular
  nonlinearity.
\newblock {\em Mediterr. J. Math.}, 17(4):Paper No. 110, 18, 2020.

\bibitem{GMnonloc}
Prashanta Garain and Tuhina Mukherjee.
\newblock Quasilinear nonlocal elliptic problems with variable singular
  exponent.
\newblock {\em Commun. Pure Appl. Anal.}, 19(11):5059--5075, 2020.

\bibitem{GRbook}
Marius Ghergu and Vicen\c{t}iu~D. R\u{a}dulescu.
\newblock {\em Singular elliptic problems: bifurcation and asymptotic
  analysis}, volume~37 of {\em Oxford Lecture Series in Mathematics and its
  Applications}.
\newblock The Clarendon Press, Oxford University Press, Oxford, 2008.

\bibitem{GST}
Jacques Giacomoni, Ian Schindler, and Peter Tak\'{a}\v{c}.
\newblock Sobolev versus {H}\"{o}lder local minimizers and existence of
  multiple solutions for a singular quasilinear equation.
\newblock {\em Ann. Sc. Norm. Super. Pisa Cl. Sci. (5)}, 6(1):117--158, 2007.

\bibitem{Stam}
David Kinderlehrer and Guido Stampacchia.
\newblock {\em An introduction to variational inequalities and their
  applications}, volume~88 of {\em Pure and Applied Mathematics}.
\newblock Academic Press, Inc. [Harcourt Brace Jovanovich, Publishers], New
  York-London, 1980.

\bibitem{Lieb}
Elliott~H. Lieb.
\newblock Sharp constants in the {H}ardy-{L}ittlewood-{S}obolev and related
  inequalities.
\newblock {\em Ann. of Math. (2)}, 118(2):349--374, 1983.

\bibitem{Ling}
Erik Lindgren and Peter Lindqvist.
\newblock Fractional eigenvalues.
\newblock {\em Calc. Var. Partial Differential Equations}, 49(1-2):795--826,
  2014.

\bibitem{PLin2}
Peter Lindqvist.
\newblock On the equation {${\rm div}\,(|\nabla u|^{p-2}\nabla
  u)+\lambda|u|^{p-2}u=0$}.
\newblock {\em Proc. Amer. Math. Soc.}, 109(1):157--164, 1990.

\bibitem{Mazya}
Vladimir Maz'ya.
\newblock {\em Sobolev spaces with applications to elliptic partial
  differential equations}, volume 342 of {\em Grundlehren der Mathematischen
  Wissenschaften [Fundamental Principles of Mathematical Sciences]}.
\newblock Springer, Heidelberg, augmented edition, 2011.

\bibitem{MS}
Tuhina Mukherjee and Konijeti Sreenadh.
\newblock On {D}irichlet problem for fractional {$p$}-{L}aplacian with singular
  non-linearity.
\newblock {\em Adv. Nonlinear Anal.}, 8(1):52--72, 2019.

\bibitem{Diaz}
Anh~Dao Nguyen, Jes\'{u}s~Ildefonso D\'{\i}az, and Quoc-Hung Nguyen.
\newblock Fractional {S}obolev inequalities revisited: the maximal function
  approach.
\newblock {\em Atti Accad. Naz. Lincei Rend. Lincei Mat. Appl.},
  31(1):225--236, 2020.

\bibitem{OP2}
Francescantonio Oliva and Francesco Petitta.
\newblock On singular elliptic equations with measure sources.
\newblock {\em ESAIM Control Optim. Calc. Var.}, 22(1):289--308, 2016.

\bibitem{OP3}
Francescantonio Oliva and Francesco Petitta.
\newblock Finite and infinite energy solutions of singular elliptic problems:
  existence and uniqueness.
\newblock {\em J. Differential Equations}, 264(1):311--340, 2018.

\bibitem{OP1}
Luigi Orsina and Francesco Petitta.
\newblock A {L}azer-{M}c{K}enna type problem with measures.
\newblock {\em Differential Integral Equations}, 29(1-2):19--36, 2016.

\bibitem{Otani}
Mitsuharu \^{O}tani.
\newblock Existence and nonexistence of nontrivial solutions of some nonlinear
  degenerate elliptic equations.
\newblock {\em J. Funct. Anal.}, 76(1):140--159, 1988.

\bibitem{PolyaSzego}
G.~P\'{o}lya and G.~Szeg\"{o}.
\newblock {\em Isoperimetric {I}nequalities in {M}athematical {P}hysics}.
\newblock Annals of Mathematics Studies, No. 27. Princeton University Press,
  Princeton, N. J., 1951.

\bibitem{Scoste}
Laurent Saloff-Coste.
\newblock {\em Aspects of {S}obolev-type inequalities}, volume 289 of {\em
  London Mathematical Society Lecture Note Series}.
\newblock Cambridge University Press, Cambridge, 2002.

\bibitem{Swanson}
Charles~A. Swanson.
\newblock The best {S}obolev constant.
\newblock {\em Appl. Anal.}, 47(4):227--239, 1992.

\bibitem{Talenti}
Giorgio Talenti.
\newblock Best constant in {S}obolev inequality.
\newblock {\em Ann. Mat. Pura Appl. (4)}, 110:353--372, 1976.

\end{thebibliography}
\end{document}